\documentclass[12pt,a4paper]{amsart}

\usepackage{amsmath,amsfonts,amssymb}

\usepackage[hmargin=2cm,vmargin=2cm]{geometry}

\usepackage{hyperref}
\usepackage{nameref,zref-xr}                    
\usepackage[all]{xy}           

\newcommand{\mbP}{\mathbb P}
\newcommand{\mbZ}{\mathbb Z}
\newcommand{\mbC}{\mathbb C}

\newcommand{\oM}{\overline{\mathcal M}}

\def\oM{{\overline{\mathcal{M}}}}

\def\mbQ{{\mathbb Q}}
\def\d{{\partial}}

\renewcommand{\>}{\right>}
\newcommand{\eps}{\varepsilon}

\newcommand{\cA}{\mathcal A}
\newcommand{\hcA}{\widehat{\mathcal A}}
\newcommand{\DR}{\mathrm{DR}}

\newcommand{\even}{\mathrm{even}}

\newcommand{\Coef}{\mathrm{Coef}}

\newcommand{\Deg}{\mathrm{Deg}}

\newcommand{\Ch}{\mathrm{Ch}}

\newcommand{\gl}{\mathrm{gl}}

\newcommand{\hR}{\widehat{R}}
\newcommand{\tQ}{\widetilde{Q}}

\DeclareMathOperator{\res}{res}
\newcommand{\ext}{\mathrm{ext}}

\newcommand{\mcC}{\mathcal{C}}

\newcommand{\mcF}{\mathcal{F}}

\newcommand{\mcL}{\mathcal{L}}
\newcommand{\vir}{\mathrm{vir}}

\newcommand{\mfc}{\mathfrak{c}}
\newcommand{\mbE}{\mathbb{E}}

\newcommand{\hT}{\widehat{T}}
\newcommand{\DK}{\mathrm{DK}}
\newcommand{\tw}{\widetilde{w}}
\newcommand{\hX}{\widehat{X}}
\newcommand{\valpha}{\vec{\alpha}}
\DeclareMathOperator{\odeg}{\overline{\mathrm{deg}}}
\newcommand{\tB}{\widetilde{B}}
\newcommand{\tb}{\widetilde{b}}
\newcommand{\ev}{\mathrm{ev}}
\newcommand{\hL}{\widehat{L}}
\newcommand{\mcP}{\mathcal{P}}
\newcommand{\vd}{\mathrm{vd}}
\newcommand{\Td}{\mathrm{Td}}
\newcommand{\rk}{\mathrm{rk}}
\newcommand{\tcL}{\widetilde{\mathcal{L}}}
\newcommand{\mcD}{\mathcal{D}}

\newcommand{\tcC}{\widetilde{\mathcal{C}}}
\newcommand{\norm}{\mathrm{norm}}

\newcommand{\tpi}{\widetilde{\pi}}
\newcommand{\tG}{\widetilde{G}}
\newcommand{\irr}{\mathrm{irr}}
\newcommand{\mcO}{\mathcal{O}}
\newcommand{\tf}{\widetilde{f}}
\newcommand{\tc}{\widetilde{c}}

\newtheorem{theorem}{Theorem}[section]
\newtheorem{proposition}[theorem]{Proposition}
\newtheorem{lemma}[theorem]{Lemma}

\newtheorem{conjecture}[theorem]{Conjecture}

\newtheorem{definition}{Definition}[section]

\newtheorem{remark}[theorem]{Remark}

\numberwithin{equation}{section}

\begin{document}

\title{Extended $r$-spin theory in all genera and the discrete KdV hierarchy}

\author{Alexandr Buryak}
\address{A. Buryak:\newline 
Faculty of Mathematics, National Research University Higher School of Economics, \newline
6 Usacheva str., Moscow, 119048, Russian Federation;\smallskip\newline 
Center for Advanced Studies, Skolkovo Institute of Science and Technology, \newline
1 Nobel str., Moscow, 143026, Russian Federation}
\email{aburyak@hse.ru}

\author{Paolo Rossi}
\address{P.~Rossi:\newline Dipartimento di Matematica ``Tullio Levi-Civita'', Universit\`a degli Studi di Padova,\newline
Via Trieste 63, 35121 Padova, Italy}
\email{paolo.rossi@math.unipd.it}

\begin{abstract}
In this paper we construct a family of cohomology classes on the moduli space of stable curves generalizing Witten's $r$-spin classes. They are parameterized by a phase space which has one extra dimension and in genus $0$ they correspond to the extended $r$-spin classes appearing in the computation of intersection numbers on the moduli space of open Riemann surfaces, while when restricted to the usual smaller phase space, they give in all genera the product of the top Hodge class by the $r$-spin class. They do not form a cohomological field theory, but a more general object which we call F-CohFT, since in genus $0$ it corresponds to a flat F-manifold. For $r=2$ we prove that the partition function of such F-CohFT gives a solution of the discrete KdV hierarchy. Moreover the same integrable system also appears as its double ramification hierarchy.
\end{abstract}

\date{\today}

\maketitle

\tableofcontents

\section{Introduction}

The moduli space of $r$-spin curves, parameterizing stable curves $C$ with $n$ marked points $p_1,\ldots,p_n\in C$ together with an $r$-th root, $r\geq 2$, of the twisted canonical bundle $\omega_C(\sum_{i=1}^n (1-\alpha_i)p_i)$, $\alpha_i \in \mbZ$, has a rich geometric structure and has recently proved to be a central tool in the study of the cohomology of the moduli space of stable curves $\oM_{g,n}$.\\

Witten's $r$-spin classes~\cite{Wit93,PV01,Chi06,Moc06,FJR13} are cohomological field theories on $\oM_{g,n}$, constructed out of the geometry of $r$-spin moduli spaces, and their intersection theory was shown in \cite{FSZ10} to be controlled by the Gelfand-Dickey $(r-1)$-KdV hierarchy, as previously conjectured in \cite{Wit93}.\\

As an example of the richness of such objects, in \cite{PPZ15}, $3$-spin classes were used to produce and prove a large system of relations in the tautological subring of the cohomology ring of~$\oM_{g,n}$. This system contains all other previously known relations between the generators of the tautological ring and is in fact conjectured to be a complete system, thus yielding an explicit description of the tautological ring itself.\\

The present paper starts with the observation that, in genus $0$, from results of \cite{JKV01} and~\cite{BCT17}, the construction of $r$-spin cohomological field theories can be extended to systems of cohomology classes on $\oM_{0,n}$ with phase space $1$-dimensionally bigger and satisfying a system of axioms corresponding to the structure of a flat F-manifold \cite{Ma05}, as opposed to the usual Frobenius manifold associated to genus $0$ CohFTs.\\

Given also the relevance of such genus $0$ extended $r$-spin classes in describing the intersection theory of moduli spaces of Riemann surfaces with boundary \cite{PST14,ST,Tes15}, first observed in \cite{BCT17}, we were motivated to look for a higher genus generalization of extended $r$-spin classes.\\

As we said, already in genus $0$, the extended $r$-spin theory is not quite a cohomological field theory/Frobenius manifold. However the notion of flat F-manifold can be naturally extended to higher genus as a system of cohomology classes on $\oM_{g,n}$ satisfying a set of axioms that is correspondingly weaker than those of a CohFT. We introduce this general notion in Section~\ref{section:F-CohFT} and we call it an F-cohomological field theory.\\

We then proceed to construct a specific (homogeneous) F-CohFT that reduces, in genus $0$, to the extended $r$-spin classes. The construction is based on a generalization of a formula of J. Gu\'er\'e \cite{Gue17}. Gu\'er\'e considers two elements in the $K$-theory of the moduli space of $r$-spin curves: the derived push-forward of the universal $r$-th root bundle from the universal curve to the moduli space and the Hodge bundle. Then he constructs a certain combination of characteristic classes of such two objects, depending on a parameter and well defined in the $K$-theory, and shows that, in a certain limit of the parameter and after push-forward to $\oM_{g,n}$, this explicit formula recovers the product of the $r$-spin class by the top Chern class of the Hodge bundle.\\

It turns out that such construction can be generalized to the extended phase space of extended $r$-spin theory giving, for generic value of the parameter, an actual cohomological field theory which reduces, in the above limit, to a homogeneous (with respect to a natural extension of the grading for $r$-spin theory) F-CohFT generalizing to all genera the extended $r$-spin theory. When restricted to the usual phase space, this F-CohFT gives, as prescribed by Gu\'er\'e's formula, the product of the top Chern class of the Hodge bundle by the $r$-spin class, which is of course a partial cohomological field theory. However such factorization does not happen on the extra component of the extended phase space.\\

We then study the problem of explicit description of the intersection theory with our F-CohFT in the case $r=2$. Note that the CohFT, depending on a parameter, discussed above, is semisimple, which means that, by the results of \cite{BPS12}, the corresponding Dubrovin-Zhang hierarchy \cite{DZ05} exists and that, in the limit, it will become a homogeneous system of evolutionary PDEs. The corresponding Hamiltonian structure, however, degenerates in the limit. Using homogeneity, we manage to completely identify such system of PDEs as an extension of the discrete (or $q$-difference) KdV hierarchy of \cite{Fre96} thereby proving a Witten-Kontsevich type result for the F-CohFT: the partition function of the extended $2$-spin theory satisfies the extended discrete KdV hierarchy. This effectively computes all intersection numbers of the F-CohFT with psi-classes.\\

We remark here that the discrete KdV hierarchy also appeared in \cite{BCR12,BCRR14}, together with its bigraded generalizations, as the natural candidate for the Dubrovin-Zhang hierarchy of the equivariant Gromov-Witten theory of local $\mbP^1$-orbifolds. It would be natural to investigate the relation between our construction of the extended $r$-spin classes and such Gromov-Witten theory.\\

Finally we prove that the double ramification hierarchy construction and results of \cite{Bur15a,BR16} can be generalized to F-CohFTs and that, for $r=2$ the DR hierarchy also corresponds to the extended discrete KdV hierarchy, the two incarnations being related by a Miura transformation, providing yet another example of DR/DZ equivalence along the lines of what was conjectured and investigated in \cite{Bur15a,BDGR16a,BDGR16b,BGR17}, but in the more general context of F-CohFTs.\\

We conclude with some remarks and a conjecture about the possible relation of our extended $r$-spin theory with open Hodge integrals, i.e. intersection numbers of psi-classes with Hodge classes on the moduli space of open Riemann surfaces, generalizing the genus $0$ results of~\cite{BCT17}.

\subsection*{Acknowledgements}

We would like to thank Andrea Brini, Guido Carlet, Oleg Chalykh, Allan Fordy, Paolo Lorenzoni, Alexander Mikhailov, Jake Solomon, Ran Tessler and Dimitri Zvonkine for useful discussions.\\

A. B. was supported by the grants RFBR-20-01-00579 and RFBR-16-01-00409.\\[1cm]

%%%%%%%%%%%%%%%%%%%%%%%%%%%%%%%%%%%%%%%%%%%%%%%%%%%%%%%%%%%%%%%%
%%%%%%%%%%%%%%%%%%%%%%%%%%%%%%%%%%%%%%%%%%%%%%%%%%%%%%%%%%%%%%%%

\section{CohFTs and F-CohFTs}\label{section:F-CohFT}
Consider the Deligne-Mumford moduli space $\oM_{g,n}$ of genus $g$ stable curves with $n$ marked points, defined for $g,n\geq 0$ and $2g-2+n>0$. In this section we describe a generalization of the notions of cohomological field theory (or CohFT) \cite{KM94} and of partial cohomological field theory \cite{LRZ15} (whose definition, we recall, is the same as for a  CohFT, but without the gluing axiom at non-separating nodes).

\subsection{F-CohFTs and flat F-manifolds}
\begin{definition}\label{definition:F-CohFT}
For $2g-1+n>0$, an F-cohomological field theory (or F-CohFT) is a system of linear maps $c_{g,n+1}\colon V^*\otimes V^{\otimes n} \to H^\even(\oM_{g,n+1},\mbC)$, where $V$ is an arbitrary vector space, together with a special element $e_1\in V$, called the unit, such that, chosen any basis $e_1,\ldots,e_{\dim V}$ of V and the dual basis $e^1,\ldots,e^{\dim V}$ of $V^*$, the following axioms are satisfied:
\begin{itemize}
\item[(i)] the maps $c_{g,n+1}$ are equivariant with respect to the $S_n$-action permuting the $n$ copies of~$V$ in $V^*\otimes V^{\otimes n}$ and the last $n$ marked points in $\oM_{g,n+1}$, respectively.
\item[(ii)] $\pi^* c_{g,n+1}(e^{\alpha_0}\otimes \otimes_{i=1}^n e_{\alpha_i}) = c_{g,n+2}(e^{\alpha_0}\otimes \otimes_{i=1}^n  e_{\alpha_i}\otimes e_1)$ for $1 \leq\alpha_0,\alpha_1,\ldots,\alpha_n\leq \dim V$, where $\pi\colon\oM_{g,n+2}\to\oM_{g,n+1}$ is the map that forgets the last marked point.\\
Moreover $c_{0,3}(e^{\alpha}\otimes e_\beta \otimes e_1) = \delta^\alpha_\beta$ for $1\leq \alpha,\beta\leq \dim V$.
\item[(iii)] $\gl^* c_{g_1+g_2,n_1+n_2+1}(e^{\alpha_0}\otimes \otimes_{i=1}^n e_{\alpha_i}) = c_{g_1,n_1+2}(e^{\alpha_0}\otimes \otimes_{i\in I} e_{\alpha_i} \otimes e_\mu) c_{g_2,n_2+1}(e^{\mu}\otimes \otimes_{j\in J} e_{\alpha_j})$ for $1 \leq\alpha_0,\alpha_1,\ldots,\alpha_n\leq \dim V$, where $I \sqcup J = \{2,\ldots,n+1\}$, $|I|=n_1$, $|J|=n_2$, and $\gl\colon\oM_{g_1,n_1+2}\times\oM_{g_2,n_2+1}\to \oM_{g_1+g_2,n_1+n_2+1}$ is the corresponding gluing map.
\end{itemize}
\end{definition}
\noindent In the above definition and in what follows, we will use Einstein's convention of sum over repeated Greek indices.

Clearly, an F-cohomological field theory is a generalization of a partial cohomological field theory. Indeed, consider an arbitrary partial CohFT $c_{g,n}\colon V^{\otimes n}\to H^\even(\oM_{g,n},\mbC)$ with a metric~$\eta$ on the phase space~$V$. Choose a basis $e_1,\ldots,e_{\dim V}\in V$ and define linear maps 
$$
c^\bullet_{g,n+1}\colon V^*\otimes V^{\otimes n}\to H^\even(\oM_{g,n+1},\mbC) 
$$ 
by
$$
c^\bullet_{g,n+1}(e^{\alpha_0}\otimes\otimes_{i=1}^n e_{\alpha_i}):=\eta^{\alpha_0\mu}c_{g,n+1}(e_\mu\otimes\otimes_{i=1}^n e_{\alpha_i}).
$$
Obviously, the maps $c^\bullet_{g,n+1}$ form an F-CohFT.

In the same way as the genus $0$ part of a CohFT produces a Frobenius manifold, the rational part of an F-CohFT produces a flat F-manifold (which explains the terminology we chose). We recall here the following facts from \cite{Ma05}, see also \cite{AL18}.

\begin{definition}
A flat F-manifold $(M,\nabla,\circ,e_1)$ is the datum of a (smooth or analytic) manifold $M$, an affine connection $\nabla\colon\mathfrak{X}(M) \otimes \mathfrak{X}(M)\to\mathfrak{X}(M)$, where $\mathfrak{X}(M)$ is the Lie algebra of vector fields on $M$, an algebra structure $(T_pM,\circ,e_1)$ with unit $e_1$ on each tangent space, (smoothly or analytically) depending on the point $p\in M$, such that the one-parameter family of connections $\nabla - \lambda \circ$ is flat and torsionless for any $\lambda\in \mbC$, and $\nabla e_1=0$.
\end{definition}

From flatness and torsionlessness of $\nabla-\lambda\circ$ one can deduce commutativity and associativity of the algebras $(T_pM,\circ,e_1)$ and, if we denote by $c(X,Y):=X\circ Y$, one can deduce that the expression $\nabla_Z c(X,Y)$ is symmetric in $X,Y,Z \in \mathfrak{X}(M)$.

If one choses flat coordinates $v^\alpha$, $\alpha=1,\ldots,\dim M$, for the F-manifold, with $e_1 = \frac{\d}{\d v^1}$, it is easy to see that locally the flat F-manifold structure can be encoded entirely in a (smooth or analytic) vector potential $\mcF^\alpha(v^*)$, $1\leq\alpha\leq \dim M$, satisfying
\begin{align*}
\frac{\d^2\mcF^\alpha}{\d v^1\d v^\beta} &= \delta^\alpha_\beta, && 1\leq \alpha,\beta\leq \dim M,\\
\frac{\d^2 \mcF^\alpha}{\d v^\beta \d v^\mu} \frac{\d^2 \mcF^\mu}{\d v^\gamma \d v^\delta} &= \frac{\d^2 \mcF^\alpha}{\d v^\gamma \d v^\mu} \frac{\d^2 \mcF^\mu}{\d v^\beta \d v^\delta}, && 1\leq \alpha,\beta,\gamma,\delta\leq \dim M.
\end{align*}
In particular the components $c^\alpha_{\beta\gamma}$ of the tensor $c$ (i.e. the structure functions of the algebra $(T_pM,\circ,e_1)$) can be written as $c^\alpha_{\beta\gamma} = \frac{\d^2 F^\alpha}{\d v^\beta \d v^\gamma}$, $1\leq \alpha,\beta,\gamma\leq \dim M$.

Given an F-CohFT $c_{g,n+1}\colon V^*\otimes V^{\otimes n} \to H^\even(\oM_{g,n+1},\mbC)$, a vector potential on $V$ satisfying precisely these equations can be constructed as the following generating function,
$$
\mcF^\alpha(v^*) = \sum_{\substack{n\geq 2\\1\leq\alpha_1,\ldots,\alpha_n\leq \dim V}} \int_{\oM_{0,n+1}} c_{0,n+1}(e^\alpha\otimes \otimes_{i=1}^n e_{\alpha_i}) \prod_{i=1}^n v^{\alpha_i},
$$
thus yielding an associated flat F-manifold structure on $V$.

\subsection{Graded F-CohFTs}

Since $H^*(\oM_{g,n},\mbC)$ is a graded $\mbC$-vector space, it is natural to consider the special case of F-CohFTs for which the vector spaces $V$ and $V^*$ are also graded,  $\deg e_1 = 0$, and the pairing has degree $0$, i.e. $\deg e^\alpha = - \deg e_\alpha$ for a  homogeneous basis $e_1,\ldots,e_{\dim V}$ of $V$. In this case the F-CohFT is called {\it graded} (or {\it homogeneous}, or {\it conformal}) if the maps $c_{g,n+1}:V^*\otimes V^{\otimes n} \to H^\even(\oM_{g,n},\mbC)$ are linear homogeneous of degree $\deg c_{g,n+1}$.

Because of axiom (ii) in Definition \ref{definition:F-CohFT}, $\deg c_{g,n+1}$ does not depend on $n$ and $\deg c_{0,n+1} =0$ for any $n\geq 2$.
Moreover, because of axiom (iii), $\deg c_{g,n+1}$ is a linear function of $g$, which implies that the most general form of grading compatible with the axioms of F-CohFT is
$$\deg c_{g,n+1} = \gamma g , \quad \gamma \in \mbC.$$
Notice that, for CohFTs $c_{g,n}\colon V^{\otimes n} \to H^\even(\oM_{g,n},\mbC)$ with metric $\eta$ on $V$, the analogous notion of grading has to be compatible with the extra gluing axiom at non-separating nodes $\gl^*c_{g,n}(\otimes_{i=1}^n e_{\alpha_i}) = c_{g-1,n+2}(\otimes_{i=1}^n e_{\alpha_i}\otimes e_\mu\otimes e_\nu) \eta^{\mu\nu}$, where $\gl\colon\oM_{g-1,n+2}\to \oM_{g,n}$ and, if the map $\eta: V^{\otimes 2}\to \mbC$ has degree $\deg \eta = -\delta$, this imposes the further condition $\gamma=\delta$, i.e.
$$\deg c_{g,n} = \delta (g-1), \qquad \deg \eta =-\delta,\qquad \delta \in \mbC, \quad 1\leq\alpha\leq\dim V.$$

%%%%%%%%%%%%%%%%%%%%%%%%%%%%%%%%%%%%%%%%%%%%%%%%%%%%%%%%%%%%%%%%
%%%%%%%%%%%%%%%%%%%%%%%%%%%%%%%%%%%%%%%%%%%%%%%%%%%%%%%%%%%%%%%%

\section{Construction of the extended $r$-spin theory in all genera}

This is the main geometric part of the paper, where we construct a graded F-CohFT in all genera, generalizing the genus $0$ extended $r$-spin theory from~\cite{BCT17}. In Sections~\ref{subsection:r-stable spin curves} and~\ref{subsection:r-spin theory} we recall the main properties of the moduli space of $r$-stable spin curves and the distinguished cohomology class on it, called Witten's class. Then in Section~\ref{subsection:F-CohFT in genus 0} we recall the construction of the extended $r$-spin theory in genus $0$ from~\cite{BCT17}. Section~\ref{subsection:F-CohFT in all genera} is devoted to the generalization of this construction to all genera.

We fix an integer $r\ge 2$ throughout this section.

\subsection{Moduli space of $r$-stable spin curves: overview}\label{subsection:r-stable spin curves}

Here we collect main facts about $r$-stable spin curves and their moduli space. We can recommend the paper~\cite[Section 2]{CZ09} for a more detailed introduction.

An {\it orbifold curve} $C$ with marked points $p_1,\ldots,p_n$ is a connected compact algebraic curve with at most nodal singularities and with an orbifold structure at each node. We require that the marked points belong to the smooth part of~$C$. Moreover, we require that the local picture at each node is $\{xy=0\}/\mbZ_m$, for some $m\ge 1$, where the action of the group $\mbZ_m$ of $m$-th roots of unity is given by $\zeta_m\cdot(x,y)=(\zeta_m x,\zeta_m^{-1}y)$, $\zeta_m=e^{\frac{2\pi i}{m}}$. Denote by~$|C|$ the underlying coarse (or non-orbifold) curve. 

Consider an $n$-tuple of integers $\valpha=(\alpha_1,\ldots,\alpha_n)\in\mbZ^n$. An {\it $r$-spin structure} of type $\valpha$ on a marked orbifold curve $(C;p_1,\ldots,p_n)$ is a pair
\begin{gather}\label{r-spin curve}
\left(L\to C,\phi\colon L^{\otimes r}\stackrel{\sim}{\to}\omega_C\left(\sum_{i=1}^n(1-\alpha_i)p_i\right)\right),
\end{gather}
where $L\to C$ is an orbifold line bundle (called also a {\it spin bundle}), $\phi$ is an isomorphism and by~$\omega_C$ we denote the canonical line bundle on~$C$. The number $\alpha_i$ is called the {\it index} of~$L$ at~$p_i$. An orbifold curve $(C;p_1,\ldots,p_n)$ is called {\it $r$-stable}, if the underlying marked curve $(|C|;p_1,\ldots,p_n)$ is stable and the isotropy group is $\mbZ_r$ at every node. An $r$-stable curve with an $r$-spin structure is called an {\it $r$-stable spin curve}.

The moduli space of $r$-stable spin curves of genus $g$ with $n$ marked points and a spin structure of type~$\valpha=(\alpha_1,\ldots,\alpha_n)$ will be denoted by $\oM^r_{g,\valpha}$. It is a smooth and proper Deligne-Mumford stack, which is non-empty if and only if the stability condition $2g-2+n>0$ and the divisibility condition
\begin{gather}\label{eq:divisibility condition}
\sum_{i=1}^n\alpha_i=2g-2+n\text{ mod } r
\end{gather}
are satisfied. In this case the moduli space $\oM^r_{g,\valpha}$ has complex dimension $3g-3+n$. 

The moduli space $\oM^r_{g,\valpha}$ is equipped with the following structures:
\begin{itemize}

\item the universal curve $\pi\colon\mcC_{\valpha}\to\oM^r_{g,\valpha}$;

\item the universal line bundle $\mcL_{\valpha}\to\mcC_{\valpha}$;

\item the $K$-theoretical element $R^\bullet\pi_*\mcL_{\valpha}:=R^0\pi_*\mcL_{\valpha}-R^1\pi_*\mcL_{\valpha}\in K^0(\oM^r_{g,\valpha})$, which, if $\oM^r_{g,\valpha}$ is non-empty, has rank $\rk(R^\bullet\pi_*\mcL_{\valpha})=-\vd(r,g,\valpha)$, where
$$
\vd(r,g,\valpha):=\frac{(g-1)(r-2)+\sum_{i=1}^n(\alpha_i-1)}{r};
$$

\item the projection $p\colon\oM^r_{g,\valpha}\to\oM_{g,n}$, obtained by forgetting the $r$-spin structure and the orbifold structure on an $r$-stable spin curve $(C;p_1,\ldots,p_n)$;
 
\item the Hodge vector bundle on $\oM^r_{g,\valpha}$, which is the pull-back of the Hodge bundle on $\oM_{g,n}$ via the projection $p$. Both of them will be denoted by $\mbE$. 
 
\end{itemize}
Given an $r$-spin structure~\eqref{r-spin curve}, we can twist the line bundle $L$ by $\mathcal{O}_C(p_i)$, for some $i$: $L\mapsto L(p_i)$. As a result, we get an $r$-spin structure of type $(\alpha_1,\ldots,\alpha_{i-1},\alpha_i-r,\alpha_{i+1},\ldots,\alpha_n)$. Therefore, if $\vec{\beta}=(\beta_1,\ldots,\beta_n)$ is another $n$-tuple on integers, satisfying $\alpha_i-\beta_i=0\text{ mod }r$ for each $i$, then the transformation, described above, allows to identify the moduli spaces $\oM^r_{g,\valpha}$ and $\oM^r_{g,\vec{\beta}}$ together with the universal curves over them: $\oM^r_{g,\valpha}=\oM^r_{g,\vec{\beta}}$, $\mcC_{\valpha}=\mcC_{\vec{\beta}}$. So we can consider the universal line bundle $\mcL_{\vec{\beta}}$ as a line bundle over $\mcC_{\valpha}$.

We will often omit the index $\valpha$ in the notations $\mcC_{\valpha}$ and $\mcL_{\valpha}$, if it is clear what the indices are. 

\subsection{$r$-spin theory}\label{subsection:r-spin theory}

If $1\le\alpha_1,\ldots,\alpha_n\le r$, and the moduli space $\oM^r_{g,\valpha}$ is non-empty, then it possesses a certain natural cohomology class 
$$
c_\vir(\alpha_1,\ldots,\alpha_n)_{g,n}\in H^{2\cdot\vd(r,g,\valpha)}(\oM^r_{g,\valpha},\mbQ),
$$
which is called {\it Witten's class}. In genus $0$ the construction was first carried out in~\cite{Wit93}. In this case a simple degree calculation shows that $R^0\pi_*\mcL=0$ and, therefore, the sheaf $R^1\pi_*\mcL$ is a vector bundle of rank $\rk(R^1\pi_*\mcL)=\vd(r,0,\valpha)$. Then 
$$
c_\vir(\alpha_1,\ldots,\alpha_n)_{0,n}:=c_{\vd(r,0,\valpha)}(R^1\pi_*\mcL).
$$
Different higher genus constructions were obtained in~\cite{PV01},~\cite{Chi06},~\cite{Moc06},~\cite{FJR13}. By the result of~\cite[Theorem 3]{PPZ15}, all these constructions give the same class after the push-forward to the moduli space $\oM_{g,n}$ via the projection $p\colon\oM^r_{g,\valpha}\to\oM_{g,n}$.

The class $c_\vir(\alpha_1,\ldots,\alpha_n)_{g,n}$ satisfies the so-called {\it Ramond vanishing property}: 
$$
c_\vir(\alpha_1,\ldots,\alpha_n)_{g,n}=0,\quad \text{if $\alpha_i=r$ for some~$i$}.
$$
Let us fix a vector space $V$ of dimension $r-1$ with a basis $e_1,\ldots,e_{r-1}$ and a metric $\eta$, defined in this basis by~$\eta_{\mu\nu}:=\delta_{\mu+\nu,r}$. Define a collection of linear maps
$$
c^r_{g,n}\colon V^{\otimes n}\to H^*(\oM_{g,n},\mbQ)
$$ 
by
$$
c^r_{g,n}(\otimes_{i=1}^n e_{\alpha_i}):=
\begin{cases}
(-1)^{\vd(r,g,\valpha)}r^{1-g}p_*c_\vir(\alpha_1,\ldots,\alpha_n)_{g,n},&\text{if condition~\eqref{eq:divisibility condition} is satisfied},\\
0,&\text{otherwise}.
\end{cases}
$$
These maps together with the metric $\eta$ and the vector $e_1$, considered as unit, form a cohomological field theory, called the {\it $r$-spin cohomological field theory}. Clearly,
$$
\deg c^r_{g,n}(\otimes_{i=1}^n e_{\alpha_i})=2\cdot\vd(r,g,\valpha),
$$
if condition~\eqref{eq:divisibility condition} is satisfied. Note that in the case $r=2$ we have
$$
c^2_{g,n}(e_1^{\otimes n})=1\in H^*(\oM_{g,n},\mbQ).
$$

\begin{remark}
The definition of the $r$-spin cohomological field theory, considered in other papers, can be different from ours by a rescaling coefficient. We use the same convention, as in~\cite{PPZ15} and~\cite{BG16}.
\end{remark} 

One of the most remarkable results about the $r$-spin cohomological field theory is the so-called $r$-spin Witten's conjecture~\cite{Wit93}, proved in~\cite{FSZ10}, which describes the intersection theory with the classes $c^r_{g,n}(\otimes_{i=1}^n e_{\alpha_i})$ in terms of the Gelfand-Dickey hierarchy. 

\subsection{F-CohFT in genus $0$}\label{subsection:F-CohFT in genus 0}

Let $1\le\alpha_1,\ldots,\alpha_n\le r$ and denote $\valpha^\ext:=(0,\alpha_1,\ldots,\alpha_n)$. In~\cite{JKV01} the authors noticed that one can define Witten's class on the moduli space $\oM^r_{0,\valpha^\ext}$ in exactly the same way, as described in the previous section. Indeed, we still have the vanishing $R^0\pi_*\mcL=0$ and, therefore, one can define
$$
c_\vir(0,\alpha_1,\ldots,\alpha_n)_{0,n+1}:=c_{\vd(r,0,\valpha^\ext)}\left(R^1\pi_*\mcL\right)\in H^*(\oM^r_{0,\valpha^\ext},\mbQ).
$$
In~\cite{BCT17} the authors systematically studied the intersection theory with these classes. In~\cite[Lemma 3.3]{BCT17} they observed that under the identification $\oM^r_{0,\valpha^\ext}=\oM^r_{0,\valpha'}$, $\valpha':=(r,\alpha_1,\ldots,\alpha_n)$, we have
$$
c_\vir(0,\alpha_1,\ldots,\alpha_n)_{0,n+1}=c_{\vd(r,0,\valpha')-1}(R^1\pi_*\mcL)\in H^*(\oM^r_{0,\valpha'},\mbQ).
$$

Consider an extension $V^\ext$ of the vector space $V$, obtained by adding a basis element $e_r$: $V^\ext:=\<e_1,\ldots,e_r\>$. Denote by $e^1,\ldots,e^r\in(V^\ext)^*$ the dual basis. We can extend the linear maps $c^r_{g,n}\colon V^{\otimes n}\to H^*(\oM_{g,n},\mbQ)$ to linear maps $(V^{\ext})^{\otimes n}\to H^*(\oM_{g,n},\mbQ)$ by saying 
$$
c^r_{g,n}(\otimes_{i=1}^n e_{\alpha_i}):=0,\text{ if $\alpha_i=r$ for some $i$}.
$$
From the results of Section 3.2 in~\cite{BCT17} it follows that the collection of linear maps 
$$
c^{r,\ext}_{0,n+1}\colon (V^\ext)^*\otimes (V^\ext)^{\otimes n}\to H^*(\oM_{0,n+1},\mbQ),
$$
defined by
\begin{gather}\label{eq:F-CohFT in genus 0}
c^{r,\ext}_{0,n+1}(e^{\alpha_0}\otimes\otimes_{i=1}^n e_{\alpha_i}):=
\begin{cases}
c^r_{0,n+1}(e_{r-\alpha_0}\otimes\otimes_{i=1}^n e_{\alpha_i}),&\text{if $1\le\alpha_0\le r-1$},\\
(-1)^{\vd(r,0,\valpha')-1}r\cdot p_* c_{\vd(r,0,\valpha')-1}(R^1\pi_*\mcL),&\text{if $\alpha_0=r$},
\end{cases}
\end{gather}
where in the second case $\valpha':=(r,\alpha_1,\ldots,\alpha_n)$ and $\mcL$ is the universal line bundle on the universal curve $\pi\colon\mcC\to\oM^r_{0,\valpha'}$, form an F-CohFT. We also present a proof of this fact in the next section, where we extend this F-CohFT to all genera. The classes~\eqref{eq:F-CohFT in genus 0} were called in~\cite{BCT17} the {\it extended $r$-spin classes}.

\subsection{F-CohFT in all genera}\label{subsection:F-CohFT in all genera}

In this section we extend the genus $0$ F-CohFT~\eqref{eq:F-CohFT in genus 0} to all genera.

\subsubsection{Characteristic class $\mfc_t(V)$}

Here we mostly recall the material from~\cite[Section 3.2]{BG16}.

Let~$V$ be a complex vector bundle of rank~$k$ over a quasi-projective variety~$S$. Denote by $a_1,\ldots,a_k$ the Chern roots of $V$. Recall that the Todd class of $V$ is defined by
$$
\Td(V):=\prod_{i=1}^k\frac{a_i}{1-e^{-a_i}}\in H^*(S,\mbQ).
$$
Let $t$ be a formal variable. Define 
$$
\lambda_t(V):=\sum_{i\ge 0}(\Lambda^iV)t^i\in K^0(S)[t].
$$
Clearly,
$$
\Ch(\lambda_t(V))=\prod_{i=1}^k(1+e^{a_i}t)\in H^*(S,\mbQ)[t],
$$
where $\Ch$ denotes the Chern character. A class $\mfc_t(V)\in H^*(S,\mbQ)[t]$ is defined by
$$
\mfc_t(V):=\Ch\left(\lambda_{-t}(V^\vee)\right)\Td(V)=\prod_{i=1}^k\left[(1-e^{-a_i}t)\frac{a_i}{1-e^{-a_i}}\right]\in H^*(S,\mbQ)[t].
$$
One can immediately see that $\mfc_t(V)|_{t=1}=c_k(V)$.

Let $\tau:=1-t$ and consider the expansion of the polynomial $\mfc_t(V)$ around $t=1$:
$$
\mfc_t(V)=\sum_{1\le i\le k,\,j\ge 0}\mfc_{i,j}(V)\tau^i,\quad\mfc_{i,j}(V)\in H^{2j}(S,\mbQ).
$$
\begin{lemma}\label{lemma:mfc-class for vector bundle}
We have $\mfc_{i,j}(V)=0$, if $j<k-i$; and $\mfc_{i,k-i}(V)=c_{k-i}(V)$.
\end{lemma}
\begin{proof}
We compute $\mfc_t(V)=\prod_{i=1}^k\left[(1-e^{-a_i}(1-\tau))\frac{a_i}{1-e^{-a_i}}\right]=\prod_{i=1}^k\left(a_i+\tau\frac{a_i}{e^{a_i}-1}\right)$. It remains to note that $\frac{a_i}{e^{a_i}-1}=1+O(a_i)$ and the lemma becomes clear.
\end{proof}

The characteristic class $\mfc_t(V)$ satisfies the properties
\begin{gather*}
\mfc_t(V_1\oplus V_2)=\mfc_t(V_1)\mfc_t(V_2),\qquad\mfc_t(\mbC^k)=(1-t)^k,
\end{gather*}
where $V_1$ and $V_2$ are two vector bundles, and $\mbC^k$ denotes the trivial vector bundle of rank~$k$. Note also that the class $\mfc_t(V)$ has the form $\mfc_t(V)=1+O(t)$ and, therefore, it is invertible in the ring $H^*(S,\mbQ)[[t]]$. Therefore, the function $\mfc_t$ can be defined for an arbitrary element in the $K$-theory of $S$ as follows:
$$
\mfc_t(B-A):=\frac{\mfc_t(B)}{\mfc_t(A)}\in H^*(S,\mbQ)[[t]],
$$
where $A$ and $B$ are two vector bundles. In~\cite{Gue16} the author observed that the class $\mfc_t(B-A)$ can be expressed in the following way:
\begin{align*}
\mfc_t(B-A)=&(1-t)^{\rk(B-A)}\exp\left(\sum_{l\ge 1}s_l(t)\Ch_l(A-B)\right),\quad\text{where}\\
s_l(t):=&\frac{B_l}{l}+(-1)^l\sum_{k=1}^l(k-1)!\left(\frac{t}{1-t}\right)^k\gamma(l,k),
\end{align*}
and the numbers $\gamma(l,k)$ are defined by the generating series
$$
\sum_{l\ge 0}\gamma(l,k)\frac{z^l}{l!}:=\frac{(e^z-1)^k}{k!}.
$$
We see that $\mfc_t(B-A)$ is a rational function of the variable $t$ with coefficients in~$H^*(S,\mbQ)$ and with a unique pole at $t=1$.

Consider the Laurent series expansion of the function $\mfc_t(B-A)$ around $t=1$:
$$
\mfc_t(B-A)=\sum_{i\ge -M,\,j\ge 0}\mfc_{i,j}(B-A)\tau^i,\quad\mfc_{i,j}(B-A)\in H^{2j}(S,\mbQ),
$$
where $M$ is the order of pole of $\mfc_t(B-A)$ at $t=1$. 

\begin{lemma}\label{lemma:mfc-class, vanishing}
If $j<\rk(B-A)-i$, then $\mfc_{i,j}(B-A)=0$.
\end{lemma} 
\begin{proof}
The function $s_l(t)$ has a pole at $t=1$ of order $l$. Therefore, the coefficient of $\tau^{-i}$, $i\ge 0$, in the expansion of $\exp\left(\sum_{l\ge 1}s_l(t)\Ch_l(A-B)\right)$ around $t=1$ belongs to the space $\bigoplus_{j\ge i}H^{2j}(S,\mbQ)$, which proves the lemma.
\end{proof}
 
\subsubsection{Auxiliary cohomological field theory}

For an $n$-tuple of integers $\valpha=(\alpha_1,\ldots,\alpha_n)$ define a class $G^{r;t}_{g,n}(\alpha_1,\ldots,\alpha_n)\in H^*(\oM_{g,n},\mbQ)(t)$ by
$$
G^{r;t}_{g,n}(\alpha_1,\ldots,\alpha_n):=p_*\left(\mfc_t(-R^\bullet\pi_*\mcL)\mfc_{t^{-r}}(\mbE^\vee)\right)\in H^*(\oM_{g,n},\mbQ)(t),
$$
if condition~\eqref{eq:divisibility condition} is satisfied, and define this class to be zero, if condition~\eqref{eq:divisibility condition} is violated.

Introduce a metric $\eta^\tau$ on the space $V^\ext$ by
$$
\eta^\tau_{\mu\nu}:=
\begin{cases}
\delta_{\mu+\nu,r},&\text{if $1\le\mu,\nu\le r-1$},\\
\delta_{\mu,\nu}\tau,&\text{otherwise}.
\end{cases}
$$
Since $H^*(\oM_{0,3},\mbQ)=\mbQ$, we will sometimes consider the class $G^{r;t}_{0,3}(\alpha_1,\alpha_2,\alpha_3)$ as a rational function from~$\mbQ(t)$.
 
\begin{proposition}\label{proposition:Grt-class, properties}
The classes $G^{r;t}_{g,n}(\alpha_1,\ldots,\alpha_n)$ satisfy the following properties:
\begin{itemize}
\item[1.] For the forgetful map $f\colon\oM_{g,n+1}\to\oM_{g,n}$, forgetting the last marked point, we have $G^{r;t}_{g,n+1}(\alpha_1,\ldots,\alpha_n,1)=f^*G^{r;t}_{g,n}(\alpha_1,\ldots,\alpha_n)$.

\item[2.] For the gluing map $\gl\colon\oM_{g_1,n_1+1}\times\oM_{g_1,n_2+1}\to\oM_{g,n}$, $g=g_1+g_2$, $n=n_1+n_2$, we have
\begin{gather}\label{eq:G-class, gluing property1}
\gl^*G^{r;t}_{g,n}(\alpha_1,\ldots,\alpha_n)=r(\eta^\tau)^{\mu\nu}G^{r;t}_{g_1,n_1+1}(\alpha_1,\ldots,\alpha_{n_1},\mu)\times G^{r;t}_{g_2,n_2+1}(\alpha_{n_1+1},\ldots,\alpha_n,\nu).
\end{gather}
\item[3.] For the gluing map $\gl\colon\oM_{g-1,n+2}\to\oM_{g,n}$ we have
\begin{gather}\label{eq:G-class, gluing property2}
\gl^*G^{r;t}_{g,n}(\alpha_1,\ldots,\alpha_n)=(1-t^{-r})r(\eta^\tau)^{\mu\nu}G^{r;t}_{g-1,n+2}(\alpha_1,\ldots,\alpha_n,\mu,\nu).
\end{gather}
\item[4.] If $1\le\alpha_1,\alpha_2,\alpha_3\le r$, then we have
$$
G^{r;t}_{0,3}(\alpha_1,\alpha_2,\alpha_3)=
\begin{cases}
\frac{1}{r},&\text{if $\alpha_1+\alpha_2+\alpha_3=r+1$},\\
\frac{1-t}{r},&\text{if $\alpha_1+\alpha_2+\alpha_3=2r+1$},\\
0,&\text{otherwise}.
\end{cases}
$$
\end{itemize}
\end{proposition}
\begin{proof}
Denote
$$
\tG^{r;t}_{g,n}(\alpha_1,\ldots,\alpha_n):=\mfc_t(-R^\bullet\pi_*\mcL)\mfc_{t^{-r}}(\mbE^\vee)\in H^*(\oM^r_{g,\valpha},\mbQ)(t).
$$
Let $[n]:=\{1,\ldots,n\}$. For a subset $I\subset[n]$, $I=\{i_1,\ldots,i_{|I|}\}$, $i_1<i_2<\ldots<i_{|I|}$, we denote by~$\alpha_I$ the string $\alpha_{i_1},\ldots,\alpha_{i_{|I|}}$.

Let us prove part 1. We have the forgetful map $\tf\colon\oM^r_{g,(\alpha_{[n]},1)}\to\oM^r_{g,\valpha}$ and, since the universal line bundle on the universal curve over $\oM^r_{g,(\alpha_{[n]},1)}$ is the pull-back of the universal line bundle on the universal curve over $\oM^r_{g,\valpha}$, we get $\tf^*\tG^{r;t}_{g,n}(\alpha_{[n]})=\tG^{r;t}_{g,n+1}(\alpha_{[n]},1)$. After that it remains to note that for the projections $p\colon\oM^r_{g,\valpha}\to\oM_{g,n}$ and $p'\colon\oM^r_{g,(\alpha_{[n]},1)}\to\oM_{g,n+1}$ the relation $p'_*\tf^*=f^*p_*$ holds in cohomology. 

Let us prove part 2. Let $I_1:=\{1,\ldots,n_1\}$ and $I_2:=\{n_1+1,\ldots,n\}$. Note that condition~\eqref{eq:divisibility condition} implies that a term in the sum on the right-hand side of formula~\eqref{eq:G-class, gluing property1} vanishes unless the numbers $1\le\mu,\nu\le r$ satisfy the conditions
\begin{gather}\label{eq:relations for mu and nu}
\mu+\sum_{i=1}^{n_1}\alpha_i=2g_1-2+n_1+1\text{ mod }r, \qquad \mu+\nu=0\text{ mod }r,
\end{gather}
which determine them uniquely. 
\begin{lemma}\label{lemma:0 and r}
Under the identification $\oM^r_{g,(r,\alpha_{[n]})}=\oM^r_{g,(0,\alpha_{[n]})}$ we have $\tG^{r;t}_{g,n+1}(r,\alpha_{[n]})=(1-t)\tG^{r;t}_{g,n+1}(0,\alpha_{[n]})$.
\end{lemma}
\begin{proof}
Consider the universal curve $\pi\colon\mcC\to\oM^r_{g,(r,\alpha_{[n]})}$ and the universal line bundles $\mcL:=\mcL_{(r,\alpha_{[n]})}$ and $\tcL:=\mcL_{(0,\alpha_{[n]})}$ on $\mcC$. Denote by $\sigma_1\colon\oM^r_{g,(r,\alpha_{[n]})}\to\mcC$ the section corresponding to the first marked point and by $\Delta_1\subset\mcC$ the image of $\sigma_1$. Clearly, we have $\tcL=\mcL(\Delta_1)$, which gives the exact sequence of sheaves
$$
0\to R^0\pi_*\mcL\to R^0\pi_*\tcL\to\sigma_1^*\tcL\to R^1\pi_*\mcL\to R^1\pi_*\tcL\to 0.
$$
Applying the characteristic class $\mfc_t$, we get $\mfc_t(-R^\bullet\pi_*\mcL)=\mfc_t(-R^\bullet\pi_*\tcL)\mfc_t(\sigma_1^*\tcL)$. Note that $(\sigma_1^*\tcL)^{\otimes r}=\sigma_1^*(\tcL^{\otimes r})\cong\sigma_1^*(\omega_\pi(\Delta_1))\cong\mathcal{O}_{\oM^r_{g,(r,\alpha_{[n]})}}$, where $\omega_\pi$ is the relative canonical line bundle on~$\mcC$. Therefore, $c_1(\sigma_1^*\tcL)=0$, and, hence, $\mfc_t(\sigma_1^*\tcL)=(1-t)$, which proves the lemma.
\end{proof}

Let us recall the structure of the boundary of the moduli space~$\oM^r_{g,\valpha}$, following closely the exposition from~\cite[Section~2.3]{CZ09}. As we have already said above, a neighbourhood of a node of an $r$-stable curve $C$ is isomorphic to the quotient of the ordinary node $\{xy=0\}$ by the $\mbZ_r$-action given by $\zeta_r\cdot(x,y)=(\zeta_rx,\zeta_r^{-1}y)$. Then a spin bundle $L$ on $C$ in the neighbourhood of the node is locally isomorphic to the quotient $\left.\{(x,y,t)|xy=0\}\right/{\mbZ_r}$ with the $\mbZ_r$-action given by $\zeta_r\cdot(x,y,t)=(\zeta_r x,\zeta_r^{-1} y,\zeta_r^a t)$, for some integer $0\le a\le r-1$. The number $a$ here depends on the chosen order of the branches of~$C$ at the node. If we interchange the branches, then the number~$a$ changes to a number~$0\le b\le r-1$, uniquely determined by the relation $a+b=0\text{ mod }r$. The numbers $a$ and $b$ will be called the {\it local indices} of the spin bundle $L$ at the node.

Suppose now that the node is separating and divides our curve~$C$ into a component of genus~$g_1$ with the marking set $I_1$ and a component of genus $g_2$ with the marking set $I_2$. Note that the numbers $a, b$ are related to the numbers $\mu, \nu$ from~\eqref{eq:relations for mu and nu} by
$$
(a,b)=
\begin{cases}
(\mu,\nu),&\text{if $1\le\mu\le r-1$},\\
(0,0),&\text{if $\mu=r$}.
\end{cases}
$$
Let us now do the following procedure:
\begin{itemize}
\item[1.] Normalize the curve $C$ at the node, $C\mapsto \widetilde C=C_1\sqcup C_2$, and take the pull-back of the spin bundle~$L$ to the normalization.
\item[2.] Forget the orbifold structure at the two new marked points (this is the same as passing to the coarse space, but only locally).
\item[3.] Replace the pull-back of the spin bundle~$L$ by the sheaf of its invariant sections in a neighbourhood of the two new marked points. This sheaf turns out to be locally free at the two new marked points.
\end{itemize}
As a result, we get $r$-spin structures of types $(\alpha_{I_1},a)$ and $(\alpha_{I_2},b)$ on the $r$-stable curves~$C_1$ and~$C_2$, respectively.

Denote by $\mcD_{g_1,I_1}$ the moduli space parameterizing singular $r$-stable spin curves with a separating node dividing a curve into a component of genus~$g_1$ with the marking set~$I_1$ and a component of genus~$g_2$ with the marking set~$I_2$. Then the procedure, described above, defines a map $\mcD_{g_1,I_1}\to \oM^r_{g_1,(\alpha_{I_1},a)}\times\oM^r_{g_2,(\alpha_{I_2},b)}$ which we denote by $\mu_{g_1,I_1}$. We also have a natural map $j_{g_1,I_1}\colon\mcD_{g_1,I_1}\to\oM^r_{g,\valpha}$. These two maps are parts of the following commutative diagram:
\begin{gather*}
\xymatrix{
\tcL\ar[r]&\tcC_\mcD=\mcC_1\sqcup\mcC_2\ar[dr]^{\tpi_\mcD}\ar[rr]^\norm & & \mcC_\mcD\ar@/^0.5pc/[dl]^{\pi_\mcD} & \mcL\ar[l]\\
& \oM^r_{g_1,(\alpha_{I_1},a)}\times\oM^r_{g_2,(\alpha_{I_2},b)}\ar[dr]_{p_1\times p_2} & \mcD_{g_1,I_1}\ar[l]^{\hspace{1.75cm}\mu_{g_1,I_1}}\ar[r]_{j_{g_1,I_1}}\ar[d]^\rho \ar@/^0.5pc/[ur]^{\sigma_e} & \oM^r_{g,\valpha}\ar[d]^p &\\
& & \oM_{g_1,n_1+1}\times\oM_{g_2,n_2+1}\ar[r]^{\hspace{1.35cm}\gl} & \oM_{g,n} &
}
\end{gather*}
Here $\pi_\mcD\colon\mcC_\mcD\to\mcD_{g_1,I_1}$ is the universal curve, pulled back from $\oM^r_{g,\valpha}$, and $\mcL\to\mcC_\mcD$ is the universal line bundle. Over $\mcD_{g_1,I_1}$ we also have another universal curve $\tpi_\mcD\colon\tcC_\mcD\to\mcD_{g_1,I_1}$, given by normalization at the separating node. Let us decompose $\tcC_\mcD=\mcC_1\sqcup\mcC_2$ and denote by $\norm\colon\tcC_\mcD\to\mcC_\mcD$ the universal normalization map. Denote by $\tcL$ the universal line bundle over~$\tcC_\mcD$. Finally, by $\sigma_e$ we denote the section of the universal curve $\mcC_\mcD$ corresponding to the distinguished node of a singular curve from~$\mcD_{g_1,I_1}$.

Note that $\mcC_i$, $i=1,2$, is the pull-back via the map $\mu_{g_1,I_1}$ of the universal curve over the $i$-th component of the product $\oM^r_{g_1,(\alpha_{I_1},a)}\times\oM^r_{g_2,(\alpha_{I_2},b)}$ and that $\mu_{g_1,I_1}^*(\mbE_1\oplus\mbE_2)=j_{g_1,I_1}^*\mbE$, where~$\mbE_1$ and~$\mbE_2$ are the Hodge vector bundles over the moduli spaces $\oM^r_{g_1,(\alpha_{I_1},a)}$ and $\oM^r_{g_2,(\alpha_{I_2},b)}$, respectively. Therefore, we have
$$
\mfc_t(-R^\bullet\tpi_{\mcD *}\tcL)\mfc_{t^{-r}}(j_{g_1,I_1}^*\mbE^\vee)=\mu_{g_1,I_1}^*\left(\tG^{r;t}_{g_1,n_1+1}(\alpha_{I_1},a)\times\tG^{r;t}_{g_2,n_2+1}(\alpha_{I_2},b)\right).
$$
We have the exact sequence of sheaves on $\mcC_\mcD$
$$
0\to\mcL\to\norm_*\tcL\to\left.\mcL\right|_{\Delta_e}\to 0,
$$
where $\Delta_e$ is the image of the section $\sigma_e$. We get the long exact sequence of higher push-forwards
\begin{gather}\label{eq:G-class, gluing property1, long exact}
0\to R^0\pi_{\mcD *}\mcL\to R^0\tpi_{\mcD *}\tcL\to\sigma_e^*\mcL\to R^1\pi_{\mcD *}\mcL\to R^1\tpi_{\mcD *}\tcL\to 0.
\end{gather}

Suppose now that $1\le a\le r-1$. Then $\sigma_e^*\mcL=0$, because sections of an orbifold line bundle necessarily vanish at points where the isotropy group acts non-trivially on the fiber. Applying the characteristic class $\mfc_t$ to~\eqref{eq:G-class, gluing property1, long exact}, we get 
$$
\mu_{g_1,I_1}^*\left(\tG^{r;t}_{g_1,n_1+1}(\alpha_{I_1},a)\times\tG^{r;t}_{g_2,n_2+1}(\alpha_{I_2},b)\right)=j_{g_1,I_1}^*\tG^{r;t}_{g,n}(\alpha_{[n]}).
$$
If $a=0$, then $(\sigma_e^*\mcL)^{\otimes r}=\sigma_e^*\omega_{\pi_\mcD}=\mathcal{O}_{\mcD_{g_1,I_1}}$, which implies that $c_1(\sigma_e^*\mcL)=0$ and, therefore, 
$$
(1-t)\mu_{g_1,I_1}^*\left(\tG^{r;t}_{g_1,n_1+1}(\alpha_{I_1},0)\times\tG^{r;t}_{g_2,n_2+1}(\alpha_{I_2},0)\right)=j_{g_1,I_1}^*\tG^{r;t}_{g,n}(\alpha_{[n]}).
$$

Using Lemma~\ref{lemma:0 and r}, in all cases $0\le a\le r-1$ we can now write
$$
j_{g_1,I_1}^*\tG^{r;t}_{g,n}(\alpha_{[n]})=(\eta^\tau)^{\mu\nu}\mu_{g_1,I_1}^*\left(\tG^{r;t}_{g_1,n_1+1}(\alpha_{I_1},\mu)\times\tG^{r;t}_{g_2,n_2+1}(\alpha_{I_2},\nu)\right).
$$
Since $\deg\mu_{g_1,I_1}=1$~\cite[page 1348]{CZ09}, we get
$$
\rho_*j_{g_1,I_1}^*\tG^{r;t}_{g,n}(\alpha_{[n]})=(\eta^\tau)^{\mu\nu}G^{r;t}_{g_1,n_1+1}(\alpha_{I_1},\mu)\times G^{r;t}_{g_2,n_2+1}(\alpha_{I_2},\nu),
$$
which implies formula~\eqref{eq:G-class, gluing property1}, because $\rho_*j^*_{g_1,I_1}=\frac{1}{r}\gl^*p_*$.

The proof of part 3 of the proposition is analagous to the proof of part 2. For numbers $0\le a,b\le r-1$, satisfying $a+b=r\text{ mod }r$, we consider the moduli space $\mcD^a_{\irr}$ parameterizing singular $r$-stable spin curves with a chosen nonseparating node and an order of branches of the curve at it, such that the local indices of the spin bundle at the node are equal to $a$ and $b$. We have natural maps 
\begin{gather*}
\xymatrix{
\oM^r_{g-1,(\alpha_{[n]},a,b)} & \mcD^a_{\irr}\ar[l]_{\hspace{1cm}\mu^a_{\irr}}\ar[r]^{\hspace{-0.15cm}j^a_{\irr}} & \oM^r_{g,\valpha}.
}
\end{gather*}
Then we proceed with the same calculation with two universal curves over $\mcD^a_\irr$, as in the proof of part 2, noting that the bundles $(\mu^a_\irr)^*\mbE$ and $(j^a_\irr)^*\mbE$ are not isomorphic, but we have the short exact sequence
$$
0\to(\mu^a_\irr)^*\mbE\to(j^a_\irr)^*\mbE\stackrel{\res}{\to}\mcO_{\mcD^a_\irr}\to 0,
$$ 
where the map $\res$ is given by taking the residue of a holomorphic $1$-form on a singular curve at the node. As a result, we get
\begin{gather*}
(j^a_\irr)^*\tG^{r;t}_{g,n}(\alpha_{[n]})=
\begin{cases}
(1-t^{-r})(\mu^a_\irr)^*\tG^{r;t}_{g-1,n+2}(\alpha_{[n]},a,b),&\text{if $1\le a\le r-1$},\\
(1-t^{-r})(1-t)(\mu^a_\irr)^*\tG^{r;t}_{g-1,n+2}(\alpha_{[n]},a,b),&\text{if $a=0$}.
\end{cases}
\end{gather*}

Define $\mcD_\irr:=\bigsqcup_{a=0}^{r-1}\mcD_\irr^a$, $\mu_\irr:=\bigsqcup_{a=0}^{r-1}\mu_\irr^a$, $j_\irr:=\bigsqcup_{a=0}^{r-1}j_\irr^a$ and let $\rho\colon\mcD_\irr\to\oM_{g-1,n+2}$ be the composition of the map $\mu_\irr$ and the projection $\bigsqcup_{\substack{0\le a,b\le r-1\\a+b=0\text{ mod }r}}\oM^r_{g-1,(\alpha_{[n]},a,b)}\to\oM_{g-1,n+2}$. The proof of formula~\eqref{eq:G-class, gluing property2} is completed by using Lemma~\ref{lemma:0 and r}, the property $\deg\mu^a_{\irr}=1$ \cite[page 1348]{CZ09} and the fact that $\rho_*j_\irr^*=\frac{1}{r}\gl^*p_*$.

Let us prove part 4. We may assume that $\alpha_1+\alpha_2+\alpha_3=1\text{ mod }r$. Using Lemma~\ref{lemma:mfc-class for vector bundle} and the vanishing $R^0\pi_*\mcL=0$, we get $\int_{\oM_{0,3}}G^{r;t}_{0,3}(\alpha_1,\alpha_2,\alpha_3)=\int_{\oM_{0,3}}p_*\mfc_t(R^1\pi_*\mcL)=\frac{(1-t)^{\vd(r,0,(\alpha_1,\alpha_2,\alpha_3))}}{r}$, which proves part 4 of the proposition.
\end{proof}

We use the standard notation $\lambda_j:=c_j(\mbE)\in H^{2j}(\oM_{g,n},\mbQ)$.
\begin{lemma}\label{lemma:Grt-class, no pole}
If $1\le\alpha_1,\ldots,\alpha_n\le r$, then the function $G^{r;t}_{g,n}(\alpha_1,\ldots,\alpha_n)\in H^*(\oM_{g,n},\mbQ)(t)$ doesn't have pole at $t=1$ and, moreover,
\begin{gather}\label{eq:Grt-class, specialization}
\left.G^{r;t}_{g,n}(\alpha_1,\ldots,\alpha_n)\right|_{t=1}=(-1)^{g+\vd(r,g,\valpha)}r^{g-1}\lambda_g c^r_{g,n}(\otimes_{i=1}^n e_{\alpha_i}).
\end{gather}
\end{lemma}
\begin{proof}
If $1\le\alpha_1,\ldots,\alpha_n\le r-1$, then the lemma follows from~\cite[Theorem 3.5]{BG16}. Let us prove that the lemma is true for the classes 
\begin{gather}\label{eq:Grt-class with m rs}
G^{r;t}_{g,n+m}(\alpha_1,\ldots,\alpha_n,\underbrace{r,\ldots,r}_{\text{$m$ times}}),\quad 1\le\alpha_1,\ldots,\alpha_n\le r-1,\quad m\ge 1,
\end{gather}
by induction on $m$. The proof is different in the cases $r\ge 3$ and $r=2$.

Suppose that $r\ge 3$. Consider a map $f\colon\oM_{g,n+m}\to\oM_{g,n+m+2}$, defined as the composition of the map $\oM_{g,n+m}\to\oM_{g,n+m}\times\oM_{0,3}$ and the gluing map $\gl\colon\oM_{g,n+m}\times\oM_{0,3}\to\oM_{g,n+m+2}$, which identifies that last marked points on curves from $\oM_{g,n+m}$ and $\oM_{0,3}$. Then, by Proposition~\ref{proposition:Grt-class, properties}, 
$$
f^*G^{r;t}_{g,n+m+2}(\alpha_1,\ldots,\alpha_n,\underbrace{r,\ldots,r}_{\text{$m-1$ times}},2,r-1)=G^{r;t}_{g,n+m}(\alpha_1,\ldots,\alpha_n,\underbrace{r,\ldots,r}_{\text{$m$ times}}),
$$
which, by the induction assumption, implies that the class~\eqref{eq:Grt-class with m rs} doesn't have pole at $t=1$. If $m\ge 2$, then formula~\eqref{eq:Grt-class, specialization} also follows from the induction assumption. In the case $m=1$ formula~\eqref{eq:Grt-class, specialization} follows from the fact that $\gl^* c^r_{g,n+2}(\otimes_{i=1}^ne_{\alpha_i}\otimes e_2\otimes e_{r-1})=0$.

Suppose that $r=2$. Since the class~\eqref{eq:Grt-class with m rs} is equal to the pull-back of the class $G^{2;t}_{g,m}(2,\ldots,2)$ via the forgetful map $\oM_{g,n+m}\to\oM_{g,m}$, we only have to prove the lemma for the classes $G^{2;t}_{g,m}(2,\ldots,2)$, $m\ge 1$, where we may assume that $m$ is even. Consider the gluing map $\gl\colon\oM_{g,m}\to\oM_{g+1,m-2}$. Then, by Proposition~\ref{proposition:Grt-class, properties}, we have
$$
\gl^* G^{2;t}_{g+1,m-2}(\underbrace{2,\ldots,2}_{\text{$m-2$ times}})=2(1-t^{-2})G^{2;t}_{g,m}(\underbrace{2,\ldots,2}_{\text{$m-2$ times}},1,1)-2\frac{t+1}{t^2}G^{2;t}_{g,m}(\underbrace{2,\ldots,2}_{\text{$m$ times}}),
$$
which, by the induction assumption, implies that the function $G^{2;t}_{g,m}(2,\ldots,2)$ doesn't have pole at $t=1$. In order to prove formula~\eqref{eq:Grt-class, specialization}, we should also use that $\gl^*\lambda_g=0$. This completes the proof of the lemma.
\end{proof}

This lemma implies that we can consider the class $G^{r;1-\tau}_{g,n}(\alpha_1,\ldots,\alpha_n)$ as an element of the space $H^*(\oM_{g,n},\mbQ)[[\tau]]$. Let $\sigma$ be a formal variable. Introduce maps
$$
C^{r;\sigma,\tau}_{g,n}\colon (V^\ext)^{\otimes n}\to H^*(\oM_{g,n},\mbQ)[[\sigma,\sigma^{-1},\tau]]
$$ 
by
$$
C^{r;\sigma,\tau}_{g,n}(\otimes_{i=1}^n e_{\alpha_i}):=\sigma^{-(g+\vd(r,g,\valpha))}r^{1-g}(-\sigma)^{\frac{1}{2}\Deg}G^{r;1-\sigma\tau}_{g,n}(\alpha_1,\ldots,\alpha_n),
$$
where $\Deg\colon H^*(\oM_{g,n},\mbQ)\to H^*(\oM_{g,n},\mbQ)$ is the linear operator, which acts on a subspace $H^m(\oM_{g,n},\mbQ)$ by the multiplication by~$m$. 

\begin{proposition}\label{proposition:C-class -- partial CohFT}
The classes $C^{r;\sigma,\tau}_{g,n}(\otimes_{i=1}^n e_{\alpha_i})$ satisfy the following properties:
\begin{itemize}
\item[1.] The class $C^{r;\sigma,\tau}_{g,n}(\otimes_{i=1}^n e_{\alpha_i})$ belongs to the space $ H^*(\oM_{g,n},\mbQ)[[\sigma,\tau]]$ and its expansion in the variables $\sigma$ and $\tau$ has the form
$$
C^{r;\sigma,\tau}_{g,n}(\otimes_{i=1}^n e_{\alpha_i})=\lambda_g c^r_{g,n}(\otimes_{i=1}^n e_{\alpha_i})+\sum_{\substack{k\ge 0,\,l\ge 1\\g+\vd(r,g,\valpha)+k-l\ge 0}}\sigma^k\tau^l C^r_{g,n,k,l}(\otimes_{i=1}^n e_{\alpha_i}),
$$
where $\deg C^r_{g,n,k,l}(\otimes_{i=1}^n e_{\alpha_i})=2\left(g+\vd(r,g,\valpha)+k-l\right)$.
\item[2.] The maps $C^{r;\sigma,\tau}_{g,n}$ form a partial CohFT with the phase space $V^\ext$ and the metric $\eta^\tau$, and with the following factorization property along the gluing map $\gl\colon\oM_{g-1,n+2}\to\oM_{g,n}$:
\begin{gather}\label{eq:loop gluing property}
\gl^*C^{r;\sigma,\tau}_{g,n}(\otimes_{i=1}^n e_{\alpha_i})=\frac{1-(1-\sigma\tau)^{-r}}{\sigma}(\eta^\tau)^{\mu\nu}C^{r;\sigma,\tau}_{g-1,n+2}(\otimes_{i=1}^n e_{\alpha_i}\otimes e_\mu\otimes e_\nu).
\end{gather}
\end{itemize}
\end{proposition}
\begin{proof}
From Lemma~\ref{lemma:mfc-class, vanishing} it follows that the coefficient of $\tau^i$, $0\le i\le g+\vd(r,g,\valpha)$, in the power series $G^{r;1-\tau}_{g,n}(\alpha_1,\ldots,\alpha_n)$ belongs to the space $\bigoplus_{j\ge g+\vd(r,g,\valpha)-i}H^{2j}(\oM_{g,n},\mbQ)$. Together with formula~\eqref{eq:Grt-class, specialization} this implies part 1 of the proposition. Part 2 follows from Proposition~\ref{proposition:Grt-class, properties}.
\end{proof}

\begin{remark}
Proposition~\ref{proposition:C-class -- partial CohFT} implies that the rescaled classes
\begin{gather}\label{rescaled classes}
\left(\frac{1-(1-\sigma\tau)^{-r}}{\sigma}\right)^{-g}C^{r;\sigma,\tau}_{g,n}(\otimes_{i=1}^n e_{\alpha_i})
\end{gather}
form a cohomological field theory with the phase space $V^\ext$ and the metric $\eta^\tau$.
\end{remark}

\subsubsection{Construction of maps $c^{r,\ext}_{g,n+1}$}

Consider the F-CohFT $C^{\bullet,r;\sigma,\tau}_{g,n+1}$ associated to the partial CohFT $C^{r;\sigma,\tau}_{g,n}$. Since $\left.C^{r;\sigma,\tau}_{g,n}(\otimes_{i=1}^n e_{\alpha_i})\right|_{\tau=0}=0$, if at least one of the $\alpha_i$'s is equal to $r$, we have
$$
C^{\bullet,r;\sigma,\tau}_{g,n+1}(e^{\alpha_0}\otimes\otimes_{i=1}^n e_{\alpha_i})\in H^*(\oM_{g,n+1},\mbQ)[[\sigma,\tau]].
$$
Define linear maps 
$$
c^{r,\ext}_{g,n+1}\colon (V^\ext)^*\otimes(V^\ext)^{\otimes n}\to H^*(\oM_{g,n+1},\mbQ)
$$
by
\begin{gather}\label{eq:definition of the extended r-spin theory}
c^{r,\ext}_{g,n+1}:=\left.C^{\bullet,r;\sigma,\tau}_{g,n+1}\right|_{\sigma=\tau=0}.
\end{gather}

\begin{theorem}\label{theorem:main properties of the extended theory}
We have the following properties:
\begin{enumerate}
\item[1.] The maps $c^{r,\ext}_{g,n+1}$ form an F-CohFT with the phase space $V^\ext$.
\item[2.] Definition~\eqref{eq:definition of the extended r-spin theory} agrees with definition~\eqref{eq:F-CohFT in genus 0} in genus~$0$.
\item[3.] The degree of the class $c^{r,\ext}_{g,n+1}(e^{\alpha_0}\otimes\otimes_{i=1}^n e_{\alpha_i})$ is given by
\begin{gather}\label{eq:degree of the extended class}
\deg c^{r,\ext}_{g,n+1}(e^{\alpha_0}\otimes\otimes_{i=1}^n e_{\alpha_i})=2\left(\frac{g(2r-2)-(\alpha_0-1)+\sum_{i=1}^n(\alpha_i-1)}{r}\right).
\end{gather}
\item[4.] If $1\le \alpha_0\le r-1$, then
$$
c^{r,\ext}_{g,n+1}(e^{\alpha_0}\otimes\otimes_{i=1}^n e_{\alpha_i})=\lambda_g c^r_{g,n+1}(e_{r-\alpha_0}\otimes\otimes_{i=1}^n e_{\alpha_i}).
$$
\item[5.] Consider the gluing map $\gl\colon\oM_{g-1,n+3}\to\oM_{g,n+1}$. Then we have
\begin{gather}\label{eq:loop property for extended}
\gl^*c^{r,\ext}_{g,n+1}(e^{\alpha_0}\otimes\otimes_{i=1}^n e_{\alpha_i})=-r\cdot c^{r,\ext}_{g-1,n+3}(e^{\alpha_0}\otimes\otimes_{i=1}^n e_{\alpha_i}\otimes e_r\otimes e_r).
\end{gather}
\item[6.] The class $c^{r,\ext}_{g,n+2}(e^r\otimes e_r\otimes\otimes_{i=1}^n e_{\alpha_i})$ is invariant with respect to the map $\oM_{g,n+2}\to\oM_{g,n+2}$, induced by the permutation of the first two marked points.
\end{enumerate}
\end{theorem}
\begin{proof}
Parts 1 and 6 of the theorem are obvious from definition~\eqref{eq:definition of the extended r-spin theory}. Parts~3 and~4 follow from part 1 of Proposition~\ref{proposition:C-class -- partial CohFT}. Part 5 follows from equation~\eqref{eq:loop gluing property}, because $\left.\frac{1-(1-\sigma\tau)^{-r}}{\sigma\tau}\right|_{\sigma=\tau=0}=-r$. 

It remains to prove part 2 of the theorem. Since $R^0\pi_*\mcL=0$, by Lemma~\ref{lemma:mfc-class for vector bundle}, we have
$$
\mfc_t(-R^\bullet\pi_*\mcL)=\mfc_t(R^1\pi_*\mcL)=\sum_{i=0}^{\vd(r,0,\valpha)}\sum_{j\ge\vd(r,0,\valpha)-i}\mfc_{i,j}(R^1\pi_*\mcL)\tau^i, 
$$
where $\mfc_{i,\vd(r,0,\valpha)-i}(R^1\pi_*\mcL)=c_{\vd(r,0,\valpha)-i}(R^1\pi_*\mcL)$. Therefore,
\begin{align*}
\left.C^{r;\sigma,\tau}_{0,n}(\otimes_{i=1}^n e_{\alpha_i})\right|_{\sigma=0}=&\left.\left[r\sum_{i=0}^{\vd(r,0,\valpha)}\sum_{j\ge\vd(r,0,\valpha)-i}(-1)^j p_*\mfc_{i,j}(R^1\pi_*\mcL)\sigma^{j-\vd(r,0,\valpha)+i}\tau^i\right]\right|_{\sigma=0}=\\
=&r\sum_{i=0}^{\vd(r,0,\valpha)}(-1)^{\vd(r,0,\valpha)-i} p_*c_{\vd(r,0,\valpha)-i}(R^1\pi_*\mcL)\tau^i,
\end{align*}
which proves part 2 of the theorem.
\end{proof}

Note that property 3 from this theorem means that the F-CohFT $c^{r,\ext}_{g,n+1}$ is graded, $\deg c^{r,\ext}_{g,n+1}=\gamma g$, with
$$
\deg e_\alpha=2\frac{\alpha-1}{r},\qquad \gamma=2\frac{2r-2}{r}.
$$

%%%%%%%%%%%%%%%%%%%%%%%%%%%%%%%%%%%%%%%%%%%%%%%%%%%%%%%%%%%
%%%%%%%%%%%%%%%%%%%%%%%%%%%%%%%%%%%%%%%%%%%%%%%%%%%%%%%%%%%

\section{Intersection theory with the classes $c^{2,\ext}_{g,n+1}$ and the discrete KdV hierarchy}

In this section we prove that the intersection theory with the extended $2$-spin classes in all genera is controlled by the discrete KdV hierarchy. In Sections~\ref{subsection:differential polynomials} and~\ref{subsection:shift operators} we list the main facts about the spaces of differential polynomials and shift operators and also recall the definition of the discrete KdV hierarchy. The main results are formulated in Section~\ref{subsection:main results} and the proofs are given in Section~\ref{subsection:proofs}.

\subsection{Differential polynomials}\label{subsection:differential polynomials}

Let $N\ge 1$ and consider formal variables $w^\alpha_d$, where $1\le\alpha\le N$ and $d\ge 0$. We will use the notations $w^\alpha:=w^\alpha_0$, $w^\alpha_x:=w^\alpha_1, w^\alpha_{xx}:=w^\alpha_2, \ldots$. The ring~$\cA_{w^1,\ldots,w^N}$ of {\it differential polynomials} is defined as the ring of polynomials in the variables~$w^\alpha_i, i>0$, with coefficients in the ring of formal power series in the variables $w^\alpha=w^\alpha_0$. We can differentiate a differential polynomial with respect to $x$ by applying the operator $\partial_x := \sum_{i\geq 0} w^\alpha_{i+1}\frac{\d}{\d w^\alpha_i}$. Let us introduce a grading $\deg w^\alpha_i=i$ and define $\cA^{[k]}_{w^1,\ldots,w^N}$ as the subspace of $\cA_{w^1,\ldots,w^N}$ of degree~$k$.

Introduce a new formal variable~$\eps$ with $\deg\eps = -1$. Let $\hcA_{w^1,\ldots,w^N}:=\cA_{w^1,\ldots,w^N}[[\eps]]$ and define~$\hcA^{[k]}_{w^1,\ldots,w^N}$ as the subspace of $\hcA_{w^1,\ldots,w^N}$ of degree~$k$. Elements of $\hcA_{w^1,\ldots,w^N}$ will be also called differential polynomials. We denote by $\hcA^{\ev}_{w^1,\ldots,w^N}\subset\hcA_{w^1,\ldots,w^N}$ the subring formed by differential polynomials, which contain only even powers of~$\eps$.

\subsection{Ring of shift operators and the discrete KdV hierarchy}\label{subsection:shift operators}

Let $\Lambda:=e^{i\eps\d_x}$ and consider the space of formal operators of the form
\begin{gather}\label{eq:general shift operator}
A=\sum_{n=-\infty}^ma_n\Lambda^n,\quad a_n\in\hcA_{w^1,\ldots,w^N},\quad m\in\mbZ.
\end{gather}
The multiplication in this space is given by 
$$
(f\Lambda^m)\cdot(g\Lambda^n):=f(e^{im\eps\d_x}g)\Lambda^{m+n},\quad f,g\in\hcA_{w^1,\ldots,w^N},\quad m,n\in\mbZ.
$$
The resulting ring is called the {\it ring of shift operators}.

For an operator~\eqref{eq:general shift operator} let
$$
A_+:=\sum_{n=0}^ma_n\Lambda^n.
$$
It is not hard to see that if $m\ge 2$ and $a_m=1$, then there exists a unique operator $B$ of the form
$$
B=\Lambda+\sum_{n=-\infty}^0 b_n\Lambda^n,
$$
such that $B^m=A$. We will denote it by $A^{\frac{1}{m}}:=B$.

Let $u$ and $v$ be formal variables and consider the ring of shift operators with coefficients from~$\hcA_{u,v}$. In this ring, we consider the operator
$$
L=\Lambda^2+\frac{v+\Lambda v}{2}\Lambda+u.
$$
The {\it discrete KdV hierarchy} is the system of evolutionary partial differential equations with the dependent variables $u$ and $v$, given by
\begin{gather}\label{eq:discrete KdV hierarchy}
\frac{\d L}{\d\tau_d}=(i\eps)^{-1}\frac{2^d}{(2d+1)!!}\left[L^{d+\frac{1}{2}}_+,L\right],\quad d\ge 0.
\end{gather}
In particular, $\frac{\d u}{\d\tau_d}=0$ and
\begin{gather}\label{eq:first flow of discrete KdV}
\frac{\d v}{\d\tau_0}=-\frac{1}{4}\d_xR\left(v^2-4u\right),\quad\text{where}\quad R:=\frac{2}{i\eps\d_x}\frac{\Lambda-1}{\Lambda+1}.
\end{gather}

Hierarchy~\eqref{eq:discrete KdV hierarchy} was first introduced in~\cite{Fre96}. Equation~\eqref{eq:first flow of discrete KdV} can be considered as a system of differential equations for the infinite sequence of functions $f_n(\tau_0)$, with fixed constants $\alpha_n$, $n\in\mbZ$, given by $f_n(\tau_0)=v(x,\tau_0)|_{x=in\eps}$, $\alpha_n=u(x,\tau_0)|_{x=in\eps}$. This system appeared in the literature under the name {\it dressing chain} (see e.g.~\cite{VS93}).

\subsection{Main results}\label{subsection:main results}

We denote by $\psi_i\in H^2(\oM_{g,n},\mbQ)$ the first Chern class of the contangent line bundle over~$\oM_{g,n}$ attached to the $i$-th marked point. Introduce formal power series
\begin{align*}
F^\alpha(t^*_*,\eps):=&\sum_{g,n\ge 0}\frac{\eps^{2g}}{n!}\sum_{\substack{1\le\alpha_1,\ldots,\alpha_n\le 2\\d_1,\ldots,d_n\ge 0}}\left(\int_{\oM_{g,n+1}}c^{2,\ext}_{g,n+1}(e^\alpha\otimes\otimes_{i=1}^n e_{\alpha_i})\prod_{i=1}^{n}\psi_{i+1}^{d_i}\right)\prod_{i=1}^n t^{\alpha_i}_{d_i},\quad \alpha=1,2,\\
w^\alpha(t^*_*,\eps):=&\frac{\d F^\alpha}{\d t^1_0}.
\end{align*}
Here we define the expression $\left(\int_{\oM_{g,n}}\cdot\right)$ to be zero, if the stability condition $2g-2+n>0$ is violated. We will always identify $t^1_0=x$. Note that the vector potential of the F-manifold, corresponding to the extended $2$-spin theory, is given by~\cite[Lemma 3.8]{BCT17}
\begin{gather}\label{eq:vector potential for extended 2-spin}
\mcF^1(v^1,v^2)=\frac{(v^1)^2}{2},\qquad \mcF^2(v^1,v^2)=v^1v^2-\frac{(v^2)^3}{12}.
\end{gather}

\begin{theorem}\label{theorem:discrete KdV}
The formal powers series
\begin{gather}\label{eq:series u and w}
u(t^*_*,\eps):=\frac{e^{i\eps\d_x/2}-e^{-i\eps\d_x/2}}{i\eps\d_x}w^1(t^*_*,\eps),\qquad v(t^*_*,\eps):=\frac{e^{i\eps\d_x/2}-e^{-i\eps\d_x/2}}{i\eps\d_x}w^2(t^*_*,\eps)
\end{gather}
satisfy the discrete KdV hierarchy after the identification $\tau_d=t^2_d$.
\end{theorem}

This theorem doesn't say anything about the flows $\frac{\d}{\d t^1_d}$. The equations for them are described by the following two statements. Let us introduce an additional grading $\odeg$ in the ring~$\hcA_{u,v}$ by
$$
\odeg u_d:=2,\qquad \odeg v_d:=1,\qquad \odeg\eps:=0.
$$

\begin{theorem}\label{theorem:existence of the extended discrete KdV}
1. There exists a unique sequence of flows $\frac{\d}{\d t^1_d}$, $d\ge 0$, of the form
\begin{align*}
\frac{\d u}{\d t^1_d}=&\frac{u^d}{d!}u_x,\\
\frac{\d v}{\d t^1_d}=&\d_x\DK_{1,d},\quad \DK_{1,d}\in\hcA^{[0]}_{u,v},
\end{align*}
which commute with the flows of the discrete KdV hierarchy and satisfy the properties
$$
\DK_{1,0}=v,\qquad\odeg\DK_{1,d}=2d+1,\qquad\left.\DK_{1,d}\right|_{\substack{\eps=0\\u=0}}=\frac{v^{2d+1}}{(-2)^d(2d+1)!!}.
$$
2. Moreover, we have
\begin{align*}
\DK_{1,1}=&-\frac{v^3}{24}+\frac{1}{2}uv-\frac{1}{8}R\left[R^{-1}v\cdot(v^2-4u)\right]-\frac{\eps^2}{32}R\left[v\cdot\d_x^2 R(v^2-4u)\right].
\end{align*}
\end{theorem}

The infinite system of flows $\frac{\d}{\d t^2_d}$ of the discrete KdV hierarchy together with the new flows~$\frac{\d}{\d t^1_d}$, given by the last theorem, will be called the {\it extended discrete KdV hierarchy}. 

\begin{theorem}\label{theorem:extended discrete KdV}
The formal power series $u(t^*_*,\eps)$ and $v(t^*_*,\eps)$, given by~\eqref{eq:series u and w}, satisfy the extended discrete KdV hierarchy.
\end{theorem}

\begin{remark}
According to our knowledge, the existence of an extra infinite sequence of flows, commuting with the flows of the discrete KdV hierarchy, wasn't known before. An interesting problem is to find an explicit description of them. 
\end{remark}

\begin{remark}
We expect that for $r\ge 3$ the intersection theory with the extended $r$-spin classes is controlled by an extension of the discrete version of the corresponding higher Gelfand-Dickey hierarchy.
\end{remark}

\subsection{Proof of the main results}\label{subsection:proofs}

We split the proof of Theorems~\ref{theorem:discrete KdV}--\ref{theorem:extended discrete KdV} in several steps. In Section~\ref{subsubsection:step 1} we prove that the power series $w^\alpha(t^*_*,\eps)$ satisfy a system of evolutionary PDEs whose right-hand sides are differential polynomials with a certain homogeneity property. We call this system the hierarchy of topological type for the extended $2$-spin theory. We also determine explicitly the equations for $\frac{\d w^1}{\d t^\beta_d}$. In Section~\ref{subsubsection:step 2} we prove that the commutativity of the flows~$\frac{\d}{\d t^2_0}$ and~$\frac{\d}{\d t^1_1}$ of the hierarchy of topological type allows to determine them uniquely up to a certain rescaling parameter~$\theta\in\mbC$ and a differential operator~$X$ of the form $X=1+\sum_{g\ge 1}X_g(\eps\d_x)^{2g}$, $X_g\in\mbC$. In Section~\ref{subsubsection:step 3} we compute $\theta$ and $X$, using a relation in the cohomology group $H^{4g}(\oM_{g,1},\mbQ)$, found in~\cite{LP11}. Then in Section~\ref{subsubsection:step 4} we prove a lemma that allows to reconstruct uniquely all the higher flows of both the hierarchy of topological type and the discrete KdV hierarchy starting from the flow~$\frac{\d}{\d t^2_0}$. Finally, after the preliminary work, done in Sections~\ref{subsubsection:step 1}--\ref{subsubsection:step 4}, we prove Theorems~\ref{theorem:discrete KdV}--\ref{theorem:extended discrete KdV} in Section~\ref{subsubsection:step 5}.

\subsubsection{Step 1: hierarchy of topological type for the extended $2$-spin theory}\label{subsubsection:step 1}

Introduce the constants
$$
L_g:=\int_{\oM_{g,3}}\lambda_g\psi_1^{2g},\quad g\ge 0,
$$ 
and the differential operator $L:=1+\sum_{g\ge 1}L_g(\eps\d_x)^{2g}$. We know that~\cite[Theorem 2]{FP00}
$$
L=\frac{i\eps\d_x}{e^{i\eps\d_x/2}-e^{-i\eps\d_x/2}}=1+\frac{\eps^2}{24}\d_x^2+O(\eps^4).
$$
We define the degrees of the variables $w^1_n$ and $w^2_n$ to be
$$
\odeg w^1_n:=2,\qquad \odeg w^2_n:=1.
$$
\begin{proposition}\label{proposition:hierarchy of topological type for extended}
The formal power series $w^\alpha(t^*_*,\eps)$ satisfy a system of PDEs of the form
\begin{gather}\label{eq:hierarchy of topological type for extended}
\frac{\d w^\alpha}{\d t^\beta_d}=\d_x Q^\alpha_{\beta,d},\quad Q^\alpha_{\beta,d}\in\hcA^{\ev,[0]}_{w^1,w^2},\quad 1\le\alpha,\beta\le 2,\quad d\ge 0,
\end{gather}
where the differential polynomials $Q^\alpha_{\beta,d}$ have the following properties:
\begin{align*}
&Q^1_{1,d}=\frac{1}{(d+1)!}L\left((L^{-1}w^1)^{d+1}\right),&& Q^1_{2,d}=0,\\
&Q^2_{1,0}=w^2,&& \left.Q^2_{2,0}\right|_{\eps=0}=-\frac{1}{4}(w^2)^2+w^1,&& \left.Q^2_{\beta,d}\right|_{\substack{\eps=0\\w^1=0}}=\frac{(w^2)^{2d+\beta}}{(-2)^{d+\beta-1}(2d+\beta)!!},\\
&\odeg Q^2_{\beta,d}=2d+\beta.
\end{align*}
\end{proposition}
\begin{proof}
We begin with the following lemma.

\begin{lemma}
The partial cohomological field theory $C^{2;\sigma,\tau}_{g,n}$ is semisimple at the origin.
\end{lemma}
\begin{proof}
Our partial CohFT defines the algebra structure on the tangent space~$T_0 V^\ext$ with the multiplication, given by $e_1\cdot e_i=e_i$ and $e_2\cdot e_2=\tau e_1$, (see part 4 of Proposition~\ref{proposition:Grt-class, properties}). This algebra, considered as an algebra over the algebraic closure of the field of fractions of $\mbC[[\sigma,\tau]]$, is clearly semisimple.
\end{proof}

Introduce formal variables $\tw^1$, $\tw^2$ and let
\begin{align*}
F^{\sigma,\tau}(t^*_*,\eps):=&\sum_{g,n\ge 0}\frac{\eps^{2g}}{n!}\sum_{\substack{1\le\alpha_1,\ldots,\alpha_n\le 2\\d_1,\ldots,d_n\ge 0}}\left(\int_{\oM_{g,n}}C^{2;\sigma,\tau}_{g,n}(\otimes_{i=1}^n e_{\alpha_i})\prod_{i=1}^{n}\psi_{i}^{d_i}\right)\prod_{i=1}^n t^{\alpha_i}_{d_i}\in\mbC[[t^*_*,\eps,\sigma,\tau]],\\
\tw^\alpha(t^*_*,\eps,\sigma,\tau):=&(\eta^\tau)^{\alpha\mu}\frac{\d^2 F^{\sigma,\tau}}{\d t^\mu_0\d t^1_0}\in\mbC[[t^*_*,\eps,\sigma,\tau]],\quad \alpha=1,2,
\end{align*}
and $\tw^\alpha_n(t^*_*,\eps,\sigma,\tau):=\d_x^n\tw^\alpha(t^*_*,\eps,\sigma,\tau)$. The string equation for $F^{\sigma,\tau}$,
$$
\frac{\d F^{\sigma,\tau}}{\d t^1_0}=\sum_{n\ge 0}t^\gamma_{n+1}\frac{\d F^{\sigma,\tau}}{\d t^\gamma_n}+\frac{(t^1_0)^2}{2}+\tau\frac{(t^2_0)^2}{2}+C\eps^2,\quad C\in\mbC,
$$
implies that
$$
\tw^\alpha_n(t^*_*,\eps,\sigma,\tau)=t^\alpha_n+\delta^{\alpha,1}\delta_{n,1}+O((t^*_*)^2)+O(\eps^2),
$$
where by $O((t^*_*)^2)$ we denote a formal power series, which doesn't contain linear and constant terms in the variables~$t^\gamma_q$. Therefore, there exists a unique formal power series $\tQ^\alpha_{\beta,d}\in\mbC[[(\tw^\gamma_q-\delta^{\gamma,1}\delta_{q,1}),\eps^2,\sigma,\tau]]$ such that $\left.\tQ^\alpha_{\beta,d}\right|_{\tw^\gamma_q=\tw^\gamma_q(t^*_*,\eps,\sigma,\tau)}=(\eta^\tau)^{\alpha\mu}\frac{\d^2 F^{\sigma,\tau}}{\d t^\mu_0\d t^\beta_d}$. Then the power series~$\tw^\alpha(t^*_*,\eps,\sigma,\tau)$ satisfy the system of PDEs 
$$
\frac{\d\tw^\alpha}{\d t^\beta_d}=\d_x\tQ^\alpha_{\beta,d},\quad 1\le\alpha,\beta\le 2,\quad d\ge 0.
$$

On the other hand, we can consider the rescaled classes~\eqref{rescaled classes}, which form a cohomological field theory. By the previous lemma, this cohomological field theory is semisimple at the origin and, therefore, by the result of~\cite{BPS12}, there exists the corresponding Dubrovin-Zhang hierarchy (also called the hierarchy of topological type). This means that after the rescaling $\eps\mapsto\eps\sqrt{\frac{1-(1-\sigma\tau)^{-2}}{\sigma}}^{-1}$ the functions $\tQ^\alpha_{\beta,d}$ become differential polynomials of differential degree $0$. This implies that
$$
\tQ^\alpha_{\beta,d}\in\hcA^{\ev,[0]}_{\tw^1,\tw^2}[[\sigma,\tau]].
$$

Let 
$$
Q^\alpha_{\beta,d}:=\left.\tQ^\alpha_{\beta,d}\right|_{\substack{\sigma=\tau=0\\\tw^\gamma_n=w^\gamma_n}}\in\hcA^{\ev,[0]}_{w^1,w^2}.
$$
Since $\left.\tw^\alpha(t^*_*,\eps,\sigma,\tau)\right|_{\sigma=\tau=0}=w^\alpha(t^*_*,\eps)$, we see that the series $w^\alpha(t^*_*,\eps)$ satisfy the system of PDEs $\frac{\d w^\alpha}{\d t^\beta_d}=\d_x Q^\alpha_{\beta,d}$. Clearly, we have $\left.Q^\alpha_{\beta,d}\right|_{w^\gamma_n=w^\gamma_n(t^*_*,\eps)}=\frac{\d F^\alpha}{\d t^\beta_d}$. 

Recall that
\begin{gather}\label{eq:extended and lambda_g}
c^{2,\ext}_{g,n+1}(e^1\otimes\otimes_{i=1}^n e_{\alpha_i})=
\begin{cases}
\lambda_g,&\text{if $\alpha_1=\alpha_2=\ldots=\alpha_n=1$},\\
0,&\text{if $\alpha_i=2$, for some $i$}.
\end{cases}
\end{gather}
Therefore, $w^1(t^*_*,\eps)=\sum_{g,n\ge 0}\frac{\eps^{2g}}{n!}\sum_{d_1,\ldots,d_n\ge 0}\left(\int_{\oM_{g,n+2}}\lambda_g\prod_{i=1}^{n}\psi_{i+2}^{d_i}\right)\prod_{i=1}^n t^1_{d_i}$. In~\cite[Sections~5.2,~5.3]{Bur15} the first author noticed that the power series $w^1(t^*_*,\eps)$ satisfies the system of PDEs
$$
\frac{\d w^1}{\d t^1_d}=\frac{1}{(d+1)!}\d_x L\left((L^{-1}w^1)^{d+1}\right),\quad d\ge 0,
$$
and, obviously, $\frac{\d w^1(t^*_*,\eps)}{\d t^2_d}=0$. Since $w^\gamma_q(t^*_*,\eps)=t^\gamma_q+\delta^{\gamma,1}\delta_{q,1}+O((t^*_*)^2)+O(\eps^2)$, any function of the variables $t^\gamma_q$, $\eps$ can be expressed as a function of the variables $w^\gamma_q(t^*_*,\eps)$ and $\eps$ in a unique way. Therefore, $Q^1_{1,d}=\frac{1}{(d+1)!}L\left((L^{-1}w^1)^{d+1}\right)$ and $Q^1_{2,d}=0$.

Obviously, $Q^\alpha_{1,0}=w^\alpha$. Using~\eqref{eq:vector potential for extended 2-spin}, we also compute $\left.Q^2_{2,0}\right|_{\eps=0}=\left.\frac{\d\mcF^2}{\d v^2}\right|_{v^\gamma=w^\gamma}=-\frac{1}{4}(w^2)^2+w^1$.

Formula~\eqref{eq:degree of the extended class} implies that
$$
O F^\alpha=\frac{5-\alpha}{2}F^\alpha,\quad\alpha=1,2,\quad\text{where}\quad O:=\eps\frac{\d}{\d\eps}+\sum_{d\ge 0}\left(\frac{3-\mu}{2}-d\right)t^\mu_d\frac{\d}{\d t^\mu_d}.
$$
Therefore, $O w^\alpha_n(t^*_*,\eps)=\left(\frac{3-\alpha}{2}-n\right)w^\alpha_n(t^*_*,\eps)$, which gives the following expression for the operator $O$ in the variables $w^\alpha_n$:
$$
O=\eps\frac{\d}{\d\eps}+\sum_{n\ge 0}\left(\frac{3-\mu}{2}-n\right)w^\mu_n\frac{\d}{\d w^\mu_n}.
$$
We further compute $O\frac{\d F^2}{\d t^\beta_d}=\left(\frac{\beta}{2}+d\right)\frac{\d F^2}{\d t^\beta_d}$, which implies that
\begin{align*}
&\left(\eps\frac{\d}{\d\eps}+\sum_{n\ge 0}\left(\frac{3-\mu}{2}-n\right)w^\mu_n\frac{\d}{\d w^\mu_n}\right)Q^2_{\beta,d}=\left(\frac{\beta}{2}+d\right)Q^2_{\beta,d}\stackrel{Q^2_{\beta,d}\in\hcA^{[0]}_{w^1,w^2}}{\Rightarrow}\\
\Rightarrow&\left(\sum_{n\ge 0}\frac{3-\mu}{2}w^\mu_n\frac{\d}{\d w^\mu_n}\right)Q^2_{\beta,d}=\left(\frac{\beta}{2}+d\right)Q^2_{\beta,d}.
\end{align*}
The last equation gives the homogeneity condition $\odeg Q^2_{\beta,d}=2d+\beta$.

It remains to compute $\left.Q^2_{\beta,d}\right|_{\substack{\eps=0\\w^1=0}}=\left.\frac{\d F_0^2}{\d t^\beta_d}\right|_{\substack{t^1_*=t^2_{\ge 1}=0\\t^2_0=w^2}}$. This can be easily proved to be equal to $\frac{(w^2)^{2d+\beta}}{(-2)^{d+\beta-1}(2d+\beta)!!}$ by induction, using the topological recursion relation $\frac{\d^2 F^\alpha_0}{\d t^\beta_d\d t^2_0}=\frac{\d^2 F^\alpha_0}{\d t^2_0\d t^\mu_0}\frac{\d F^\mu_0}{\d t^\beta_{d-1}}$ \cite[Lemma 3.6]{BCT17} and formula~\eqref{eq:vector potential for extended 2-spin}. 
\end{proof}

System~\eqref{eq:hierarchy of topological type for extended} will be called the {\it hierarchy of topological type} for the extended $2$-spin theory.

\subsubsection{Step 2: analysis of the commutativity of the flows $\frac{\d}{\d t^2_0}$ and $\frac{\d}{\d t^1_1}$}\label{subsubsection:step 2}

Denote by $\odeg_{w^1}$ and~$\odeg_{w^2}$ the gradings with respect to the variables $w^1_n$ and $w^2_n$, respectively. Since $\odeg Q^2_{\beta,d}=2d+\beta$, we can decompose
$$
Q^2_{\beta,d}=\sum_{k=0}^{d+\left[\frac{\beta}{2}\right]}Q^2_{\beta,d,k},\quad \odeg_{w^1}Q^2_{\beta,d,k}=2k, \quad \odeg_{w^2}Q^2_{\beta,d,k}=2d+\beta-2k. 
$$
In particular, the differential polynomial $Q^2_{2,0}$ has the form 
$$
Q^2_{2,0}(w^1,w^2,\eps)=Q^2_{2,0,0}(w^2,\eps)+Xw^1,
$$
where $X$ is a differential operator of the form 
$$
X=1+\sum_{g\ge 1}X_g(\eps\d_x)^{2g},\quad X_g\in\mbC.
$$

Consider the flows $\frac{\d}{\d t^2_0}$ and $\frac{\d}{\d t^1_1}$ of the hierarchy of topological type~\eqref{eq:hierarchy of topological type for extended}:
\begin{align*}
&\frac{\d w^1}{\d t^2_0}=0, && \frac{\d w^1}{\d t^1_1}=\frac{1}{2}\d_x L\left((L^{-1}w^1)^2\right),\\
&\frac{\d w^2}{\d t^2_0}=\d_x\left(Q^2_{2,0,0}+X w^1\right), && \frac{\d w^2}{\d t^1_1}=\d_x\left(Q^2_{1,1,0}+Q^2_{1,1,1}\right).
\end{align*}
Introduce a formal variable~$w$ with degree $\odeg w_n:=1$. After the Miura transformation 
$$
w^1(u,w,\eps)=Lu,\qquad w^2(u,w,\eps)=XLw
$$
the flows $\frac{\d}{\d t^2_0}$ and $\frac{\d}{\d t^1_1}$ have the following form:
\begin{align}
&\frac{\d u}{\d t^2_0}=0, && \frac{\d u}{\d t^1_1}=u u_x,\label{eq:general system,1}\\
&\frac{\d w}{\d t^2_0}=\d_x Q(w,\eps)+u_x, && \frac{\d w}{\d t^1_1}=\d_x(C(w,\eps)+B(u,w,\eps)),\label{eq:general system,2}
\end{align}
where 
\begin{align}
&Q,C\in\hcA^{\ev,[0]}_w,&& B\in\hcA^{\ev,[0]}_{u,w},\label{eq:general property,1}\\
&Q=-\frac{1}{4}w^2+O(\eps^2),&& C=-\frac{1}{6}w^3+O(\eps^2),&& B=uw+O(\eps^2),\label{eq:general property,2}\\
&\odeg Q=2, && \odeg C=\odeg B=3.\label{eq:general property,3}
\end{align}

\begin{proposition}\label{proposition:main proposition about compatibility}
Consider arbitrary differential polynomials $Q,C\in\hcA^{\ev,[0]}_w$ and $B\in\hcA^{\ev,[0]}_{u,w}$, satisfying properties~\eqref{eq:general property,2},~\eqref{eq:general property,3}. Then the flows $\frac{\d}{\d t^2_0}$ and $\frac{\d}{\d t^1_1}$, given by~\eqref{eq:general system,1},~\eqref{eq:general system,2}, commute if and only if
\begin{align}
Q=&\left.-\frac{1}{4}(Rw)^2\right|_{\eps\mapsto\theta\eps},\label{eq:formula for Q}\\
C=&\left.\left[-\frac{1}{8}w(Rw)^2-\frac{1}{24}R^{-1}\left((Rw)^3\right)-\frac{\eps^2}{32}Rw\cdot \d_x^2R((Rw)^2)\right]\right|_{\eps\mapsto\theta\eps},\label{eq:formula for C}\\
B=&\left.\left[\frac{1}{2}uw+\frac{1}{2}R^{-1}\left(u\cdot Rw\right)+\frac{\eps^2}{8}\d_x^2 Ru\cdot Rw\right]\right|_{\eps\mapsto\theta\eps},\label{eq:formula for B}
\end{align}
for some complex parameter $\theta$.
\end{proposition}
\begin{proof}
Denote by $D_{t^\alpha_d}$ the operator of derivation along the flow $\frac{\d}{\d t^\alpha_d}$, $(\alpha,d)=(2,0),(1,1)$, of system~\eqref{eq:general system,1},~\eqref{eq:general system,2}. We see that the difference $D_{t^2_0}D_{t^1_1}w-D_{t^1_1}D_{t^2_0}w$ has the following form:
$$
D_{t^2_0}D_{t^1_1}w-D_{t^1_1}D_{t^2_0}w=P_1(w,\eps)+P_2(u,w,\eps)+P_3(u,\eps),
$$
where $\odeg_w P_1=4$, $\odeg_w P_2=2$ and $\odeg_w P_3=0$. Therefore, the flows $\frac{\d}{\d t^2_0}$ and $\frac{\d}{\d t^1_1}$ commute if and only if $P_1=P_2=P_3=0$. Let us write these equations explicitly.

For $P_3$ we have
$$
P_3=\d_x\left[B(u,u_x,\eps)-uu_x\right],
$$
which is equal to zero if and only if
\begin{align}
&B=uw+\tB(u_x,w,\eps),\quad\text{where}\label{eq:form of B,1}\\
&\tB(u,w,\eps)=\sum_{g\ge 1}\eps^{2g}\sum_{i+j=2g-1}\tb^g_{i,j}u_iw_j,\quad\tb^g_{i,j}\in\mbC,\quad\text{and}\label{eq:form of B,2}\\
&\tB(u,w,\eps)=-\tB(w,u,\eps).\label{eq:form of B,3}
\end{align}
We assume now that these equations are satisfied.
 
Let us now compute the differential polynomial $P_2$. For a differential polynomial $f\in\hcA_w$ we use the following standard notation:
$$
f_*:=\sum_{n\ge 0}\frac{\d f}{\d w_n}\d_x^n.
$$
Note that, since $\odeg C=3$, we have $C=\frac{1}{3}C_*w$. For $P_2$ we get the following expression:
\begin{align*}
P_2=&\d_x\left[C_*u_x+B(u,\d_x Q,\eps)-Q_*\d_x B(u,w,\eps)\right]=\\
=&\d_x\left[C_*u_x+\tB(u_x,\d_x Q,\eps)-Q_*\d_x\tB(u_x,w,\eps)-\sum_{n\ge 0}\sum_{i=1}^{n+1}{n+1\choose i}\frac{\d Q}{\d w_n}u_iw_{n+1-i}\right].
\end{align*}
Therefore, $P_2=0$ is and only if
\begin{gather}\label{eq:compatibility condition1}
C_*u=\sum_{n\ge 0}\sum_{i=1}^{n+1}{n+1\choose i}\frac{\d Q}{\d w_n}u_{i-1}w_{n+1-i}+Q_*\d_x\tB(u,w,\eps)-\tB(u,\d_x Q,\eps).
\end{gather}
Let us assume that this equation is satisfied.

For $P_1$ we then compute
\begin{align*}
P_1=&\d_x\left[C_*\d_x Q-Q_*\d_xC\right]=\d_x\left[C_*\d_x Q-\frac{1}{3}Q_*\d_x(C_*w)\right]=\\
=&\d_x\left[\sum_{n\ge 0}\sum_{i=1}^{n+1}{n+1\choose i}\frac{\d Q}{\d w_n}\d_x^iQ\cdot w_{n+1-i}+\frac{2}{3}Q_*\d_x\tB(\d_x Q,w,\eps)\right.\\
&\hspace{0.6cm}\left.-\frac{1}{3}Q_*\d_x\sum_{n\ge 0}\sum_{i=1}^{n+1}{n+1\choose i}\frac{\d Q}{\d w_n}w_{i-1}w_{n+1-i}\right].
\end{align*}
Therefore, $P_1=0$ if and only if
\begin{gather}\label{eq:compatibility condition2}
\sum_{\substack{n\ge 0\\1\le i\le n+1}}{n+1\choose i}\frac{\d Q}{\d w_n}\d_x^iQ\cdot w_{n+1-i}+\frac{2}{3}Q_*\d_x\tB(\d_xQ,w,\eps)-\frac{1}{3}Q_*\d_x\sum_{\substack{n\ge 0\\1\le i\le n+1}}{n+1\choose i}\frac{\d Q}{\d w_n}w_{i-1}w_{n+1-i}=0.
\end{gather}
As a result, $P_1=P_2=P_3=0$ if and only if the differential polynomial $B$ has the form~\eqref{eq:form of B,1}--\eqref{eq:form of B,3} and equations~\eqref{eq:compatibility condition1},~\eqref{eq:compatibility condition2} are satisfied.

Let us analyze formula~\eqref{eq:compatibility condition1}. Substituting it in the equation $\left(\frac{1}{3}C_*w\right)_*u-C_*u=0$, we get 
\begin{multline}\label{eq:main equation for Q and tB}
\frac{1}{3}\sum_{\substack{n,m\ge 0\\1\le i\le n+1}}\frac{\d^2 Q}{\d w_n\d w_m}{n+1\choose i}w_{i-1}w_{n+1-i}u_m+\sum_{\substack{n\ge 0\\1\le i\le n+1}}\left(\frac{1}{3}{n+2\choose i}-{n+1\choose i}\right)\frac{\d Q}{\d w_n}u_{i-1}w_{n+1-i}\\
-\frac{1}{3}\tB(w,(\d_xQ)_*u,\eps)+\frac{2}{3}\tB(u,\d_xQ,\eps)-Q_*\d_x\tB(u,w,\eps)=0.
\end{multline}
\begin{lemma}
Equation~\eqref{eq:main equation for Q and tB} determines differential polynomials $Q$ and $\tB$ of the form
\begin{align*}
Q=&-\frac{1}{4}w^2+\frac{1}{2}\sum_{g\ge 1}\eps^{2g}\sum_{i+j=2g}q^g_{i,j}w_iw_j,\quad q^g_{i,j}=q^g_{j,i},\\
\tB=&\sum_{g\ge 1}\eps^{2g}\sum_{i+j=2g-1}\tb^g_{i,j}u_iw_j,\quad \tb^g_{i,j}=-\tb^g_{j,i},
\end{align*}
uniquely up to the rescaling $\eps\mapsto\theta\eps$ for some complex number $\theta$.
\end{lemma}
\begin{proof}
Let $g\ge 1$. The coefficient of $\eps^{2g}$ on the left-hand side of equation~\eqref{eq:main equation for Q and tB} has the form
$$
M_1(u,w;q^g_{*,*},\tb^g_{*,*})+M_2(u,w;q^{\le g-1}_{*,*},\tb^{\le g-1}_{*,*}),
$$
where the coefficients of the differential polynomial $M_1$ depend on the coefficients $q^g_{i,j}$ and $\tb^g_{i,j}$ linearly. Therefore, in order to prove the lemma, it is enough to check that the linear system for the coefficients $q^g_{i,j}$ and $\tb^g_{i,j}$, given by the equation $M_1=0$, has one-dimensional space of solutions for $g=1$, and has only zero solution for $g\ge 2$.

We compute
\begin{align}
M_1=&\sum_{\substack{n+m=2g\\1\le i\le n+1}}q^g_{n,m}\left(\frac{1}{3}{n+1\choose i}w_{n+1-i}w_{i-1}u_m+\left(\frac{1}{3}{n+2\choose i}-{n+1\choose i}\right)w_m w_{n+1-i}u_{i-1}\right)\label{eq:main linear system of Q and tB}\\
&+\sum_{i+j=2g-1}\tb^g_{i,j}\left(\frac{1}{6}w_i\d_x^{j+1}(wu)-\frac{1}{6}u_i\d_x^{j+1}(w^2)+\frac{1}{2}w\d_x(u_iw_j)\right),\notag
\end{align}
and equate all the coefficients of this differential polynomial to zero. The coefficient of $ww_{2g}u$ is equal to $-\frac{2g}{3}q^g_{2g,0}+\frac{1}{3}\tb^g_{0,2g-1}$, which gives
\begin{gather}\label{eq:tb2g-1}
\tb^g_{0,2g-1}=2g\cdot q^g_{2g,0}.
\end{gather}
Taking the coefficient of $w_iw_ju$, $1\le i\le g$, $j=2g-i$, in~\eqref{eq:main linear system of Q and tB}, we get
$$
\frac{1}{6}q^g_{2g,0}{2g+2\choose i+1}-\frac{2g+1}{3}q^g_{i,j}+\frac{1}{12}(\tb_{j,i-1}^g-\tb^g_{j-1,i})-\frac{1}{6}\tb^g_{0,2g-1}{2g\choose i}=0.
$$
Taking the coefficient of $ww_iu_j$, $1\le i\le g$, $j=2g-i$, in~\eqref{eq:main linear system of Q and tB}, we obtain
$$
q^g_{0,2g}\left(\frac{1}{3}{2g+2\choose i+1}-{2g+1\choose i}\right)+q^g_{i,j}\frac{2g+1}{3}+\left(\frac{1}{3}\tb_{j-1,i}^g+\frac{1}{6}\tb_{j,i-1}^g\right)+\frac{1}{6}{2g\choose i}\tb^g_{0,2g-1}=0.
$$
Summing the last two equations, we get
$$
q_{0,2g}^g\left(\frac{1}{2}{2g+2\choose i+1}-{2g+1\choose i}\right)+\frac{1}{4}(\tb_{j,i-1}^g+\tb_{j-1,i}^g)=0,\quad 1\le i\le g,\quad j=2g-i,
$$
which allows to compute all the coefficients $\tb^g_{i,j}$ in terms of $q_{0,2g}^g$:
\begin{gather}\label{eq:tbij}
\tb_{i,j}^g=\left(2{2g\choose i+1}-2{2g\choose i}+(-1)^{i+1}(2g-2)\right)s^g_{0,2g},\quad 0\le i\le 2g-1,\quad j=2g-1-i.
\end{gather}
For the coefficients $q^g_{i,j}$ we then get 
\begin{gather}\label{eq:qij}
q^g_{i,j}=\frac{1}{2g+1}\left({2g\choose i-1}+{2g\choose i+1}-g{2g\choose i}+(-1)^{i+1}(g-1)\right)q^g_{0,2g},\quad 1\le i\le 2g-1,\quad j=2g-i.
\end{gather}

It is now easy to see that if $g=1$ and equations~\eqref{eq:tb2g-1},~\eqref{eq:tbij} and~\eqref{eq:qij} are satisfied, then $M_1=0$. Suppose $g\ge 2$. Using equations~\eqref{eq:tb2g-1},~\eqref{eq:tbij} and~\eqref{eq:qij}, we compute that the coefficient of $w_xw_{2g-2}u_x$ in $M_1$ is equal to 
$$
-\frac{(g+1)(4g^2-7g+2)}{3}q^g_{0,2g}.
$$
Since the quadratic polynomial $4g^2-7g+2$ doesn't have integer roots, we conclude that $q^g_{0,2g}=0$ and, hence, all the coefficients $q^g_{i,j}$ and $\tb^g_{i,j}$ are equal to zero. The lemma is proved.
\end{proof}

It remains to prove that for the differential polynomials $Q$, $C$ and $B$, given by formulas~\eqref{eq:formula for Q},~\eqref{eq:formula for C} and~\eqref{eq:formula for B}, the flows $\frac{\d}{\d t^2_0}$ and $\frac{\d}{\d t^1_1}$, given by system~\eqref{eq:general system,1},~\eqref{eq:general system,2}, commute. Without loss of generality we can assume that $\theta=1$. 

First, we have to check that the differential polynomial $B$ satisfies properties~\eqref{eq:form of B,1}--\eqref{eq:form of B,3}. We have
\begin{gather}\label{eq:formula for tB}
\tB(u_x,w,\eps)=B(u,w,\eps)-uw=-\frac{1}{2}uw+\frac{1}{2}R^{-1}\left(u\cdot Rw\right)+\frac{\eps^2}{8}\d_x^2 Ru\cdot Rw.
\end{gather}
In order to check the skew-symmetry of $\tB(u,w,\eps)$ we compute
\begin{gather}\label{eq:RtB}
R\tB(w_x,w_x,\eps)=-\frac{1}{4}\d_xR(w^2)+\frac{1}{2}\d_xRw\cdot w+\frac{1}{16}\eps^2\d_xR\left(\d_xRw\cdot\d_xRw\right).
\end{gather}

\begin{lemma}
We have the following identity in the ring $\hcA_{u,v}$:
\begin{gather}\label{eq:identiity with Lambda}
\frac{\Lambda-1}{\Lambda+1}\left(uv+\frac{\Lambda-1}{\Lambda+1}u\cdot\frac{\Lambda-1}{\Lambda+1}v\right)=\frac{\Lambda-1}{\Lambda+1}u\cdot v+u\cdot\frac{\Lambda-1}{\Lambda+1}v.
\end{gather}
\begin{proof}
Applying the operator $\Lambda+1$ to the both sides of~\eqref{eq:identiity with Lambda}, we come to the following equivalent identity:
\begin{align*}
&(\Lambda-1)\left(uv+\frac{\Lambda-1}{\Lambda+1}u\cdot\frac{\Lambda-1}{\Lambda+1}v\right)=(\Lambda+1)\left(\frac{\Lambda-1}{\Lambda+1}u\cdot v+u\cdot\frac{\Lambda-1}{\Lambda+1}v\right)\Leftrightarrow\\
\Leftrightarrow&\Lambda u\cdot\Lambda v+\Lambda\frac{\Lambda-1}{\Lambda+1}u\cdot \Lambda\frac{\Lambda-1}{\Lambda+1}v-uv-\frac{\Lambda-1}{\Lambda+1}u\cdot\frac{\Lambda-1}{\Lambda+1}v=\\
&\hspace{4cm}=\Lambda\frac{\Lambda-1}{\Lambda+1}u\cdot \Lambda v+\Lambda u\cdot \Lambda\frac{\Lambda-1}{\Lambda+1}v+\frac{\Lambda-1}{\Lambda+1}u\cdot v+u\cdot\frac{\Lambda-1}{\Lambda+1}v.
\end{align*}
Expressing $\frac{\Lambda-1}{\Lambda+1}u=u-\frac{2}{\Lambda+1}u$, we come to the following equation:
$$
\frac{2\Lambda}{\Lambda+1}u\cdot\frac{2\Lambda}{\Lambda+1}v-\left(2u-\frac{2}{\Lambda+1}u\right)\left(2v-\frac{2}{\Lambda+1}v\right)=0,
$$
which is clearly true.
\end{proof}
\end{lemma}

Recall that $R=\frac{2(\Lambda-1)}{i\eps\d_x(\Lambda+1)}$. Using identity~\eqref{eq:identiity with Lambda}, we then obtain $\d_xR\left(w^2-\frac{\eps^2}{4}\d_xRw\cdot\d_xRw\right)=2\d_xRw\cdot w$, which immediately implies that the right-hand side of~\eqref{eq:RtB} is equal to zero. Therefore, $\tB(u,w,\eps)=-\tB(w,u,\eps)$.

Then we need to check equations~\eqref{eq:compatibility condition1} and~\eqref{eq:compatibility condition2}. Let us begin with equation~\eqref{eq:compatibility condition1}. 
\begin{lemma}\label{lemma:identity for Q}
For $Q=-\frac{1}{4}(Rw)^2$ we have
\begin{gather}\label{first term for condition1}
\sum_{n\ge 0}\sum_{i=1}^{n+1}{n+1\choose i}\frac{\d Q}{\d w_n}u_{i-1}w_{n+1-i}=-\frac{1}{2}Rw\cdot\left(Ru\cdot w+\frac{\eps^2}{4}\d_xR(Ru\cdot\d_xRw)\right).
\end{gather}
\end{lemma}
\begin{proof}
We compute
\begin{align*}
\sum_{n\ge 0}\sum_{i=1}^{n+1}{n+1\choose i}\frac{\d Q}{\d w_n}u_iw_{n+1-i}=&Q_*\d_x(uw)-u\cdot\d_x Q=-\frac{1}{2}Rw\cdot\left(\d_xR(uw)-u\cdot\d_xRw\right)\stackrel{\text{by eq.\eqref{eq:identiity with Lambda}}}{=}\\
=&-\frac{1}{2}Rw\cdot\left(\d_xRu\cdot w+\frac{\eps^2}{4}\d_xR(\d_xRu\cdot\d_xRw)\right),
\end{align*}
which proves the lemma.
\end{proof}
The last lemma allows to compute the first term on the right-hand side of~\eqref{eq:compatibility condition1}. We compute the second term in the following way:
\begin{align}
Q_*\d_x\tB(u,w,\eps)=&Q_*\left(\tB(u_x,w,\eps)-\tB(w_x,u,\eps)\right)\stackrel{\text{by eq.\eqref{eq:formula for tB}}}{=}\notag\\
=&-\frac{1}{2}Rw\cdot\left(\frac{1}{2}u\cdot Rw+\frac{\eps^2}{8}R\left(\d_x^2Ru\cdot Rw\right)-\frac{1}{2}Ru\cdot w-\frac{\eps^2}{8}R\left(Ru\cdot\d_x^2Rw\right)\right).\label{second term for condition1}
\end{align}
By~\eqref{eq:formula for tB}, the third term on the right-hand side of~\eqref{eq:compatibility condition1} is equal to
\begin{gather}\label{third term for condition1}
-\frac{1}{4}\left[-\frac{1}{2}u\cdot(Rw)^2+\frac{1}{2}R^{-1}\left(Ru\cdot(Rw)^2\right)+\frac{\eps^2}{8}Ru\cdot\d_x^2R\left((Rw)^2\right)\right].
\end{gather}
On the other hand, we have
$$
C_*u=-\frac{1}{8}u\cdot(Rw)^2-\frac{1}{4}w\cdot Ru\cdot Rw-\frac{1}{8}R^{-1}\left(Ru\cdot(Rw)^2\right)-\frac{\eps^2}{32}Ru\cdot\d_x^2R\left((Rw)^2\right)-\frac{\eps^2}{16}Rw\cdot\d_x^2R(Ru\cdot Rw),
$$
which is equal to the sum of expressions~\eqref{first term for condition1}, \eqref{second term for condition1} and~\eqref{third term for condition1}.

Let us now check equation~\eqref{eq:compatibility condition2}. Using Lemma~\ref{lemma:identity for Q} and formula~\eqref{eq:formula for tB}, we compute the left-hand side of~\eqref{eq:compatibility condition2}:
\begin{align*}
&-\frac{1}{2}Rw\cdot\left[\d_xR(w\cdot Q)-\d_xRw\cdot Q\right]\\
&-\frac{1}{3}Rw\cdot \d_xR\left[-\frac{1}{2}w\cdot Q+\frac{1}{2}R^{-1}(Rw\cdot Q)+\frac{\eps^2}{8}Rw\cdot\d_x^2R Q\right]\\
&+\frac{1}{6}Rw\cdot\d_xR\left[-\frac{1}{2}Rw\cdot\left(w\cdot Rw+\frac{\eps^2}{4}\d_xR(\d_xRw\cdot Rw)\right)\right],
\end{align*}
which can be easily seen to be equal to zero. This completes the proof of the proposition.
\end{proof}

We see that in order to determine the differential polynomials $Q^2_{2,0}$ and $Q^2_{1,1}$ it remains to compute the operator $X$ and the parameter $\theta$. This will be done in the next section.

\subsubsection{Step 3: computation of the operator $X$ and the parameter $\theta$}\label{subsubsection:step 3}

For a differential operator~$K$ of the form
$$
K=1+\sum_{g\ge 1}K_g(\eps\d_x)^{2g},\quad K_g\in\mbC,
$$
denote
$$
\widehat{K}(z):=1+\sum_{g\ge 1}K_gz^{2g}.
$$
Introduce the constants $I_{1,g}$, $I_{2,g}$, $g\ge 0$, and the corresponding generating series $I_1(z)$, $I_2(z)$ by
\begin{align*}
&I_{1,g}:=\int_{\oM_{g,3}}c^{2,\ext}_{g,3}(e^2\otimes e_2\otimes e_1)\psi_3^{2g}, && I_1(z):=1+\sum_{g\ge 1}I_{1,g}z^{2g},\\
&I_{2,g}:=\int_{\oM_{g,3}}c^{2,\ext}_{g,3}(e^2\otimes e_2\otimes e_1)\psi_2^{2g}, && I_2(z):=1+\sum_{g\ge 1}I_{2,g}z^{2g}.
\end{align*}
In order to find power series $\hX(z)$, $I_1(z)$ and $I_2(z)$, we will find three relations between them. 

From the string equation for the functions $F^\alpha$,
$$
\frac{\d F^\alpha}{\d t^1_0}=\sum_{n\ge 0}t^\gamma_{n+1}\frac{\d F^\alpha}{\d t^\gamma_n}+t^\alpha_0,
$$
it follows that
\begin{align*}
w^\alpha_k(t^*_*,\eps)=&
\begin{cases}
\delta_{k,1}+\sum_{g\ge 0}\eps^{2g}L_g t^1_{2g+k}+O((t^*_*)^2),&\text{if $\alpha=1$},\\
\sum_{g\ge 0}\eps^{2g}I_{2,g}t^2_{2g+k}+O((t^*_*)^2),&\text{if $\alpha=2$},
\end{cases}\\
\frac{\d w^2(t^*_*,\eps)}{\d t^2_0}=&1+\sum_{g\ge 0}\eps^{2g}I_{1,g}t^1_{2g+1}+O((t^*_*)^2).
\end{align*}
Since the functions $w^\alpha(t^*_*,\eps)$ satisfy the equation $\frac{\d w^2}{\d t^2_0}=\d_x(Q^2_{2,0,0}(w^2,\eps)+Xw^1)$ and $\odeg Q^2_{2,0,0}=2$, we obtain
\begin{gather}\label{eq:first relation for hX}
I_1=\hX\cdot\hL.
\end{gather}
 
Using the dilaton equation for the functions $F^\alpha$,
$$
\frac{\d F^\alpha}{\d t^1_1}=\eps\frac{\d F^\alpha}{\d\eps}+\sum_{n\ge 0}t^\gamma_n\frac{\d F^\alpha}{\d t^\gamma_n}-F^\alpha,
$$
we get
$$
\frac{\d w^2(t^*_*,\eps)}{\d t^1_1}=\sum_{g\ge 0}\eps^{2g}(2g+1)I_{2,g}t^2_{2g}+O((t^*_*)^2).
$$
Since the functions $w^\alpha(t^*_*,\eps)$ satisfy the equation $\frac{\d w^2}{\d t^1_1}=\d_x(Q^2_{1,1,0}+Q^2_{1,1,1})$, we obtain
\begin{gather}\label{eq:tmp for the second relation for hX}
\sum_{g\ge 0}(2g+1)I_{2,g}z^{2g}=\left(\sum_{g\ge 0}z^{2g}\Coef_{\eps^{2g}w^1_xw^2_{2g}}\d_x Q^2_{1,1,1}\right)I_2(z).
\end{gather}
Denote
$$
R_\theta:=R|_{\eps\mapsto\theta\eps},\qquad T_\theta:=\sqrt{R_\theta}.
$$
Proposition~\ref{proposition:main proposition about compatibility} implies that
$$
Q^2_{1,1,1}=XL B\left(L^{-1}w^1,(XL)^{-1}w^2,\eps\right).
$$
Note that the differential polynomial $T_\theta B(u,T_\theta^{-1}w,\eps)$ has the form
\begin{align*}
T_\theta B(u,T_\theta^{-1}w,\eps)=&\frac{1}{2}T_\theta\left(u\cdot T_\theta^{-1}w\right)+\frac{1}{2}T_\theta^{-1}\left(u\cdot T_\theta w\right)+\frac{\theta^2\eps^2}{8}T_\theta\left(\d_x^2R_\theta u\cdot T_\theta w\right)=\\
=&uw+P(u,w,\eps),
\end{align*}
where $\frac{\d P}{\d u}=\frac{\d P}{\d u_x}=0$. From this we conclude that
\begin{align*}
\Coef_{\eps^{2g}w^1_xw^2_{2g}}\d_xQ^2_{1,1,1}=&\Coef_{\eps^{2g}w^1_xw^2_{2g}}\left[\d_x T_\theta^{-1}XL\left(w^1\cdot(XL)^{-1}T_\theta w^2\right)\right]=\\
=&\Coef_{z^{2g}}\left[\frac{d}{dz}\left(z\hT^{-1}_\theta\hX\hL\right)\left(\hX\hL\right)^{-1}\hT_\theta\right]\stackrel{\text{by eq.\eqref{eq:first relation for hX}}}{=}\\
=&\Coef_{z^{2g}}\left[\frac{d}{dz}\left(z\hT^{-1}_\theta I_1\right)I_1^{-1}\hT_\theta\right].
\end{align*}
Substituting this formula in~\eqref{eq:tmp for the second relation for hX}, we get $I_2+z\frac{d}{dz}I_2=\frac{d}{dz}\left(z\hT^{-1}_\theta I_1\right)I_1^{-1}\hT_\theta I_2$, which implies that $z\frac{d}{dz}\log\left(I_2\right)=z\frac{d}{dz}\log\left(\hT^{-1}_\theta I_1\right)$ and, hence,\begin{gather}\label{eq:second relation for hX}
I_2=\hT_\theta^{-1}\cdot I_1.
\end{gather}

In order to find a third relation between the power series $\hX$, $I_1$ and $I_2$, we will use a certain relation in the cohomology of the moduli space of curves, found in~\cite{LP11}. For decompositions $\{1,2,\ldots,n\}=I\sqcup J$ and $g=g_1+g_2$ denote
$$
\Delta^{g_1,g_2}_{I,J}:=\oM_{g_1,|I|+1}\times\oM_{g_2,|J|+1},
$$
where the marked points on curves from $\oM_{g_1,|I|+1}$ are labelled by the numbers from $I\cup\{n+1\}$; and the marked points on curves from $\oM_{g_2,|J|+1}$ are labelled by the numbers from $J\cup\{n+2\}$. Denote by 
$$
\iota_{I,J}^{g_1,g_2}\colon\Delta^{g_1,g_2}_{I,J}\to\oM_{g,n}
$$ 
the gluing map, which identifies the marked points labelled by~$n+1$ and~$n+2$. In~\cite[Theorem 0.1]{LP11} the authors proved the following relation in the cohomology of $\oM_{g,1}$:
$$
\psi_1^{2g}=\sum_{\substack{g_1+g_2=g\\g_1,g_2\ge 1}}\sum_{a+b=2g-1}(-1)^a\frac{g_2}{g}\left(\iota_{\{1\},\emptyset}^{g_1,g_2}\right)_*\left(\psi_2^a\psi_3^b\right)\in H^{4g}(\oM_{g,1},\mbQ),\quad g\ge 1.
$$ 
Pulling back this relation via the forgetful map $\oM_{g,n}\to\oM_{g,1}$, we obtain an analagous relation in the cohomology of $\oM_{g,n}$:
\begin{gather}\label{eq:main cohomological relation}
\psi_1^{2g}=\sum_{g_1+g_2=g}\sum_{\substack{I\sqcup J=\{1,\ldots,n\}\\1\in I\\2g_1+|I|-1>0\\2g_2+|J|-1>0}}\sum_{a+b=2g-1}(-1)^a\frac{g_2}{g}\left(\iota_{I,J}^{g_1,g_2}\right)_*\left(\psi_{n+1}^a\psi_{n+2}^b\right)\in H^{4g}(\oM_{g,n},\mbQ),\quad g,n\ge 1.
\end{gather}
Multiplying the both sides of this relation, for $n=3$, by $c^{2,\ext}_{g,3}(e_1\otimes e_2\otimes e^2)$ and integrating over~$\oM_{g,3}$ we get
\begin{align}
&z\frac{d}{dz}I_1=-I_1z\frac{d}{dz}\hL+z\frac{d}{dz}\left(I_2^2\right)-\left(\hL-1\right)z\frac{d}{dz}I_1\Rightarrow\notag\\
\Rightarrow & I_2^2=\hL\cdot I_1.\label{eq:third relation for hX}
\end{align}

Relations~\eqref{eq:first relation for hX},~\eqref{eq:second relation for hX} and~\eqref{eq:third relation for hX} imply that
\begin{align*}
I_1=&\frac{2(e^{i\theta z}-1)}{\theta(e^{iz/2}-e^{-iz/2})(e^{i\theta z}+1)},\\
I_2=&\frac{iz}{e^{iz/2}-e^{-iz/2}}\sqrt{\frac{2}{i\theta z}\frac{e^{i\theta z}-1}{e^{i\theta z}+1}}=1+\frac{1+\theta^2}{24}z^2+O(z^4),\\
\hX=&\frac{2}{i\theta z}\frac{e^{i\theta z}-1}{e^{i\theta z}+1}.
\end{align*}
In order to determine $\theta$, we compute
\begin{align*}
I_{2,1}=&\int_{\oM_{1,3}}c^{2,\ext}_{1,3}(e^2\otimes e_2\otimes e_1)\psi_2^2=\int_{\oM_{1,2}}c^{2,\ext}_{1,2}(e^2\otimes e_2)\psi_2=\\
=&\int_{\oM_{1,2}}c^{2,\ext}_{1,2}(e^2\otimes e_2)\left(\frac{1}{24}\iota^\circ_*(1)+\left(\iota^{0,1}_{\{1,2\},\emptyset}\right)_*(1)\right)\stackrel{\text{eq. \ref{eq:loop property for extended}}}{=}\frac{1}{12},
\end{align*}
where $\iota^\circ\colon\oM_{0,4}\to\oM_{1,2}$ is the gluing map. Therefore, $\theta^2=1$ and
$$
\hX=\frac{2}{iz}\frac{e^{iz}-1}{e^{iz}+1}.
$$

As a result of this and the previous sections, we obtain the following statement.

\begin{proposition}\label{proposition:flows (2,0) and (1,1)}
After the Miura transformation 
$$
u(w^1,w^2,\eps)=L^{-1}w^1,\qquad v(w^1,w^2,\eps)=L^{-1}w^2
$$
the flows $\frac{\d}{\d t^2_0}$ and $\frac{\d}{\d t^1_1}$ of the hierarchy of topological type~\eqref{eq:hierarchy of topological type for extended} are given by
\begin{align*}
\frac{\d u}{\d t^2_0}=&0,\\
\frac{\d v}{\d t^2_0}=&-\frac{1}{4}\d_xR\left(v^2-4u\right)
\end{align*}
and
\begin{align*}
\frac{\d u}{\d t^1_1}=&uu_x,\\
\frac{\d v}{\d t^1_1}=&\d_x\left[-\frac{v^3}{24}+\frac{1}{2}uv-\frac{1}{8}R\left(R^{-1}v\cdot(v^2-4u)\right)-\frac{\eps^2}{32}R\left(v\cdot\d_x^2 R(v^2-4u)\right)\right].
\end{align*}
\end{proposition}

\subsubsection{Step 4: uniqueness of higher flows}\label{subsubsection:step 4}

The following lemma is crucial for reconstruction of higher flows of the hierarchy of topological type and the extended discrete KdV hierarchy.

\begin{lemma}\label{lemma:second uniqueness lemma}
Let $Q(w,\eps)=-\frac{w^2}{4}+O(\eps)\in\hcA^{[0]}_w$, $\odeg Q=2$ and $d\ge 3$. Let us also fix a non-zero complex constant $C$. Suppose that there exists a differential polynomial $P(u,w,\eps)\in\hcA^{[0]}_{u,w}$, satisfying 
$$
P|_{\substack{\eps=0\\u=0}}=C w^d,\qquad\odeg P=d,
$$
such that the flows $\frac{\d}{\d t}$ and $\frac{\d}{\d\tau}$, given by
\begin{align*}
&\frac{\d u}{\d t}=0, && \frac{\d u}{\d\tau}=\begin{cases}0,&\text{if $d$ is even},\\ \d_x \frac{u^{(d+1)/2}}{((d+1)/2)!},&\text{if $d$ is odd},\end{cases}\\
&\frac{\d w}{\d t}=\d_xQ+u_x,&&\frac{\d w}{\d\tau}=\d_xP,
\end{align*}
commute. Then such a differential polynomial $P$ is unique. 
\end{lemma}
\begin{proof}
We begin with the following lemma which is a close analog of Lemma~2.4 in~\cite{Bur15}.
\begin{lemma}\label{lemma:first uniqueness lemma}
Let $f(w,\eps)=\frac{w^2}{2}+O(\eps)\in\hcA^{[0]}_w$ and $q_0(w)\in\mbC[[w]]$. Suppose that there exists a differential polynomial $q(w,\eps)\in\hcA^{[0]}_w$, satisfying $q(w,\eps)=q_0(w)+O(\eps)$, such that the flows $\frac{\d w}{\d t}=\d_x f$ and $\frac{\d w}{\d\tau}=\d_x q$ commute. Then such a differential polynomial $q$ is unique.
\end{lemma}
\begin{proof}
The proof is very similar to the proof of Lemma~2.4 in~\cite{Bur15}. The commutativity of the flows $\frac{\d}{\d t}$ and $\frac{\d}{\d\tau}$ means that $(\d_xq)_*\d_xf-(\d_xf)_*\d_xq=0$.  We claim that, using this equation, we can reconstruct all the coefficients of the $\eps$-expansion of $q$ term by term. In order to prove that each coefficient is determined uniquely, it is enough to show that if $q_d\in\cA^{[d]}_w$, $d\ge 1$, and 
\begin{gather}\label{eq:uniqueness 1,tmp 1}
(\d_xq_d)_*(ww_x)-(ww_x)_*\d_xq_d=0,
\end{gather}
then $q_d=0$. 

The left-hand side of~\eqref{eq:uniqueness 1,tmp 1} is equal to $\d_x\left((q_d)_*(ww_x)-w\d_xq_d\right)$ and, therefore, $(q_d)_*(ww_x)-w\d_xq_d=0$. We decompose $q_d=\sum_{\lambda\in\mcP_d}q_{d,\lambda}(w)w_\lambda$, where $\mcP_d$ is the set of all partitions of $d$, $w_\lambda:=\prod_{i=1}^{l(\lambda)}w_{\lambda_i}$ and $q_{d,\lambda}\in\mbC[[w]]$. Let us consider the lexicographical order on monomials $w_\lambda$, $\lambda\in\mcP_d$. We have
\begin{gather}\label{eq:uniqueness 1,tmp 2}
(q_{d,\lambda}w_\lambda)_*(ww_x)-w\d_x(q_{d,\lambda}w_\lambda)=\left(m_1(\lambda)+\sum_{i\ge 2}(i+1)m_i(\lambda)\right)q_{d,\lambda}w_\lambda w_x+
\begin{smallmatrix}
\text{monomials}\\\text{with lower}\\\text{lexicographical order}
\end{smallmatrix},
\end{gather}
where $m_i(\lambda):=\sharp\{1\le j\le l(\lambda)|\lambda_j=i\}$. Suppose that $q_d\ne 0$. Then choose the lexicographically maximal $\lambda$ such that $q_{d,\lambda}\ne 0$. By~\eqref{eq:uniqueness 1,tmp 2}, $(q_d)_*(ww_x)-w\d_xq_d$ is equal to the sum of the term $\left(m_1(\lambda)+\sum_{i\ge 2}(i+1)m_i(\lambda)\right)q_{d,\lambda}w_\lambda w_x$ and monomials of lower order. Since $m_1(\lambda)+\sum_{i\ge 2}(i+1)m_i(\lambda)\ne 0$, we get a contradiction, which proves the lemma.
\end{proof}

Let us now prove Lemma~\ref{lemma:second uniqueness lemma}. We decompose $P(u,w,\eps)=\sum_{i=0}^{[d/2]}P_d(u,w,\eps)$, where $\odeg_u P_i=2i$ and $\odeg_w P_i=d-2i$. Denote by $D_t$ and $D_\tau$ the operators of derivation along the flows $\frac{\d}{\d t}$ and $\frac{\d}{\d\tau}$, respectively. We then compute
\begin{align*}
&D_tD_\tau w-D_\tau D_tw=\d_x\sum_{k=0}^{[(d+1)/2]}L_k(u,w,\eps),\quad\text{where}\\
&L_k(u,w,\eps)=\sum_{n\ge 0}\left(\frac{\d P_k}{\d w_n}\d_x^{n+1}Q+\frac{\d P_{k-1}}{\d w_n}\d_x^{n+1}u-\frac{\d Q}{\d w_n}\d_x^{n+1}P_k\right)-\delta_{k,(d+1)/2}\d_x\frac{u^{[(d+1)/2]}}{[(d+1)/2]!}.
\end{align*}
Here we adopt the convention $P_i:=0$, if $i<0$ or $i>[d/2]$. Note that $\odeg_u L_k=2k$ and $\odeg_w L_k=d+1-2k$. Since the flows $\frac{\d}{\d t}$ and $\frac{\d}{\d\tau}$ commute, we obtain that $L_k=0$ for all~$k$. 

Let us prove that the equations $L_k=0$, $0\le k\le [d/2]$, determine the differential polynomials~$P_k$ uniquely. By Lemma~\ref{lemma:first uniqueness lemma}, the equation $L_0=0$ determines~$P_0$ uniquely. Suppose that $1\le k\le [d/2]$ and that we have already determined $P_{k-1}$. We have to prove that if $R\in\hcA^{[0]}_{u,w}$ satisfies the properties $\odeg_u R=2k$, $\odeg_w R=d-2k$ and
$\sum_{n\ge 0}\left(\frac{\d R}{\d w_n}\d_x^{n+1}Q-\frac{\d Q}{\d w_n}\d_x^{n+1}R\right)=0$, then $R=0$. Considering the $\eps$-expansion of $R$, it is easy to see that it is enough to show that if $r\in\cA_{u,w}^{[p]}$, $p\ge 0$, satisfies $\odeg_u r=2k$, $\odeg_w r=d-2k$ and
\begin{gather}\label{eq:uniqueness 2,tmp 1}
\sum_{n\ge 0}\frac{\d r}{\d w_n}\d_x^n(ww_x)-w\d_x r=0,
\end{gather}
then $r=0$. 

Suppose that $r\ne 0$. Let us then decompose
$$
r=\sum_{\substack{\lambda=(\lambda_1,\ldots,\lambda_k)\\\lambda_1\ge\ldots\ge\lambda_k\ge 0\\|\lambda|\le p}}r_\lambda(w)u_\lambda,\quad r_\lambda\in\cA^{[p-|\lambda|]}_w,
$$
and choose the maximal $e$ such that $r_\lambda\ne 0$ for some $\lambda$ with $|\lambda|=e$. Then 
\begin{gather}\label{eq:uniqueness 2,tmp 2}
\sum_{n\ge 0}\frac{\d r}{\d w_n}\d_x^n(ww_x)-w\d_x r=\sum_{|\lambda|=e}(-wr_\lambda\d_xu_\lambda)+\sum_{|\lambda|\le e}f_\lambda(w)u_\lambda,
\end{gather}
for some $f_\lambda\in\cA_w^{[p+1-|\lambda|]}$. The right-hand side of~\eqref{eq:uniqueness 2,tmp 2} is clearly not equal to zero, which contradicts equation~\eqref{eq:uniqueness 2,tmp 1}. The lemma is proved. 
\end{proof}

\subsubsection{Step 5: proof of Theorems \ref{theorem:discrete KdV}--\ref{theorem:extended discrete KdV}}\label{subsubsection:step 5}

Let us first prove Theorem~\eqref{theorem:discrete KdV}. By Propositions~\ref{proposition:hierarchy of topological type for extended} and~\ref{proposition:flows (2,0) and (1,1)}, the Miura transformation 
\begin{gather}\label{eq:Miura between DZ and discrete KdV}
u(w^1,w^2,\eps)=L^{-1}w^1,\qquad v(w^1,w^2,\eps)=L^{-1}w^2
\end{gather}
transforms the flows $\frac{\d}{\d t^2_d}$ of the hierarchy of topological type~\eqref{eq:hierarchy of topological type for extended} to a system of the form
\begin{gather*}
\frac{\d u}{\d t^2_d}=0,\qquad\frac{\d v}{\d t^2_d}=\d_x\tQ^2_{2,d},\quad\tQ^2_{2,d}\in\hcA^{[0]}_{u,v},
\end{gather*}
satisfying
$$
\tQ^2_{2,0}=-\frac{1}{4}R(v^2-4u),\qquad \odeg\tQ^2_{2,d}=2d+2,\qquad\left.\tQ^2_{2,d}\right|_{\substack{\eps=0\\u=0}}=\frac{v^{2d+2}}{(-4)^{d+1}(d+1)!}.
$$
On the other hand, the equations of the discrete KdV hierarchy~\eqref{eq:discrete KdV hierarchy} have the form
\begin{gather*}
\frac{\d u}{\d\tau_d}=0,\qquad\frac{\d v}{\d\tau_d}=\d_x\DK_{2,d},\quad\DK_{2,d}\in\hcA^{[0]}_{u,v},
\end{gather*}
and satisfy the properties $\DK_{2,0}=\tQ^2_{2,0}$, $\odeg\DK_{2,d}=\odeg\tQ^2_{2,d}=2d+2$. 

Let us check that $\left.\DK_{2,d}\right|_{\substack{\eps=0\\u=0}}=\frac{v^{2d+2}}{(-4)^{d+1}(d+1)!}$. Note that for any $f,g\in\cA_{u,v}$ and $m,n\ge 0$ we have
$$
(i\eps)^{-1}[f\Lambda^m,g\Lambda^n]=(mf\d_x g-ng\d_x f+O(\eps))\Lambda^{m+n}.
$$
Therefore,
\begin{align*}
\left.\d_x\DK_{2,d}\right|_{\substack{\eps=0\\u_*=0}}=&\left.\left((i\eps)^{-1}\frac{2^d}{(2d+1)!!}\Coef_\Lambda\left[(\Lambda^2+v\Lambda)^{d+1/2}_+,\Lambda^2+v\Lambda\right]\right)\right|_{\eps=0}=\\
=&\frac{2^d}{(2d+1)!!}(-v)\d_x\Coef_{z^0}\left((z^2+vz)^{d+1/2}\right)=\\
=&\frac{2^d}{(2d+1)!!}(-v)\d_x\left(\frac{(-1)^d(2d-1)!!}{2^{3d+1}d!}v^{2d+1}\right)=\\
=&\d_x\frac{v^{2d+2}}{(-4)^{d+1}(d+1)!}.
\end{align*}
Thus, $\left.\DK_{2,d}\right|_{\substack{\eps=0\\u=0}}=\left.\tQ^2_{2,d}\right|_{\substack{\eps=0\\u=0}}$ and, by Lemma~\ref{lemma:second uniqueness lemma}, $\tQ^2_{2,d}=\DK_{2,d}$, which proves Theorem~\ref{theorem:discrete KdV}.

Let us prove Theorem~\ref{theorem:existence of the extended discrete KdV}. We have just proved that after the Miura transformation~\eqref{eq:Miura between DZ and discrete KdV} the flows $\frac{\d}{\d t^2_d}$ of the hierarchy of topological type~\eqref{eq:hierarchy of topological type for extended} coincide with the flows of the discrete KdV hierarchy. The flows~$\frac{\d}{\d t^1_d}$ then give extra commuting flows for the discrete KdV hierarchy and, by Proposition~\ref{proposition:hierarchy of topological type for extended}, they satisfy the requirements from Part 1 of Theorem~\ref{theorem:existence of the extended discrete KdV}. The uniqueness of these extra flows is guaranteed by Lemma~\ref{lemma:second uniqueness lemma}. The formula for~$\DK_{1,1}$ follows from Proposition~\ref{proposition:flows (2,0) and (1,1)}.

After what we have done, Theorem~\ref{theorem:extended discrete KdV} becomes obvious, because, as we have just explained, the extended discrete KdV hierarchy coincides with the hierarchy of topological type~\eqref{eq:hierarchy of topological type for extended}, transformed by the Miura transformation~\eqref{eq:Miura between DZ and discrete KdV}.

%%%%%%%%%%%%%%%%%%%%%%%%%%%%%%%%%%%%%%%%%%%%%%%%%%%%%%%%%%%
%%%%%%%%%%%%%%%%%%%%%%%%%%%%%%%%%%%%%%%%%%%%%%%%%%%%%%%%%%%

\section{DR hierarchy for the extended $2$-spin theory}

In this section we briefly explain how to extend to the context of F-CohFT the construction of the double ramification hierarchy, traditionally associated with (partial) CohFTs. The idea is to work directly with the evolutionary PDEs (vector fields on the formal loop space, see for instance \cite{BDGR16b,Ros17}) instead of Hamiltonians and Poisson structure that are lost when passing to F-CohFTs. We then apply these constructions to this paper's main example of F-CohFT, namely the extended $2$-spin theory, computing the DR hierarchy explicitly.

\subsection{DR hierarchy for F-CohFTs}

Let $c_{g,n+1}\colon V^*\otimes V^{\otimes n} \to H^\even(\oM_{g,n+1},\mbC)$ be any F-CohFT and $\DR_g(a_1,\ldots,a_n)\in H^{2g}(\oM_{g,n},\mbQ)$, where $(a_1,\ldots,a_n)\in\mbZ^n$, be the double ramification cycle. Recall that, by Hain's formula \cite{Hai11} (and the result of~\cite{MW13}) for the DR cycle on the moduli of curves of compact type, and the vanishing of $\lambda_g$ on its complement in $\oM_{g,n}$, the cohomology class $\lambda_g\DR_g(-\sum_{j=1}^n a_j,a_1,\ldots,a_n) \in H^{4g}(\oM_{g,n+1},\mbQ)$ is a degree $2g$ homogeneous polynomial in the coefficients $a_1,\ldots,a_n$.

Consider formal variables $u^1,\ldots,u^{\dim V}$. The {\it double ramification hierarchy} is the infinite system of PDEs
\begin{gather}\label{eq:DR hierarchy}
\frac{\d u^\alpha}{\d t^\beta_d}=\d_xP^\alpha_{\beta,d},\quad 1\le \alpha,\beta\le \dim V,\quad d\ge 0,
\end{gather}
where
\begin{equation*}
P^\alpha_{\beta,d}:=\sum_{\substack{g\geq 0,\,n\geq 0\\2g+n>0\\k_1,\ldots,k_n\geq 0}} \frac{\eps^{2g}}{n!} \Coef_{(a_1)^{k_1}\ldots(a_n)^{k_n}} \left(\int_{\DR_g(-\sum_{j=1}^n a_j,0,a_1,\ldots,a_n)} \hspace{-0.8cm}\lambda_g \psi_2^d c_{g,n+2}(e^\alpha\otimes e_\beta\otimes \otimes_{j=1}^n e_{\alpha_j}) \right)\prod_{j=1}^n u^{\alpha_j}_{k_j}.
\end{equation*}
We have the following result.
\begin{theorem}
All the equations of the DR hierarchy (\ref{eq:DR hierarchy}) are compatible with each other, namely $$\frac{\d}{\d t^{\beta_2}_{d_2}}\left(\frac{\d u^\alpha}{\d t^{\beta_1}_{d_1}}\right) = \frac{\d}{\d t^{\beta_1}_{d_1}}\left(\frac{\d u^\alpha}{\d t^{\beta_2}_{d_2}}\right)$$
for any $1\leq \alpha,\beta_1,\beta_2\leq \dim V$, $d_1,d_2\geq 0$.
\end{theorem}
\begin{proof}
The proof mirrors closely the one for commutativity of the Hamiltonians for the DR hierarchy of a genuine CohFT, see \cite[Section~4]{Bur15a}. For a subset $I=\{i_1,i_2,\ldots\}\subset \{1,4,\ldots,n+3\}$, $i_1<i_2<\ldots$, $n\geq 0$, let $A_I:=(a_{i_1},a_{i_2},\ldots)$. For $I,J\subset\{1,4,\ldots,n+3\}$ with $I\sqcup J=\{1,4,\ldots,n+3\}$, and for $g_1,g_2>0$ with $g_1+|I|>0$, $g_2+|J|>0$, let us denote by $\DR_{g_1}(0_2,A_I,k)\boxtimes \DR_{g_2}(0_3,A_J,-k)$ the cycle in~$\oM_{g_1+g_2,n+3}$ obtained by gluing the two double ramification cycles at the marked points labeled by the integer~$k$. Here, by slight abuse of notation, $0_i$ indicates a coefficient $0$ at the marked point $i$ and the coefficient $a_j$, for $j\in I$ or $j \in J$, is attached to the marked point $j$. Then, for any $g,n\geq 0$,
\begin{equation*}
\sum_{\substack{I,J\subset\{1,4,\ldots,n+3\}\\ I\sqcup J=\{1,4,\ldots,n+3\}\\ k\in \mbZ,\, g_1\ge 0,\,g_2 \ge 0\\g_1+g_2 = g\\g_1+|I|,\,g_2+|J|>0}}  \lambda_g\  k\  \DR_{g_1}(0_2,A_I,k)\boxtimes \DR_{g_2}(0_3,A_J,-k)=0.
\end{equation*}
One then needs to intersect this relation with the class $ \psi_2^{d_1} \psi_3^{d_2} c_{g,n+3}(e^{\alpha_1}\otimes \otimes_{i=2}^{n+3}e_{\alpha_i})$, where as usual the covector $e^{\alpha_1}$ is attached to the marked point $1$ and each vector $e_{\alpha_i}$ is attached to the marked point $i$. Thanks to the gluing axiom of the F-CohFT and forming the corresponding generating functions, depending on whether, in the above sum, the marked point $1$ belongs to the subset $I$ or $J$, we obtain the left-hand side or minus the right-hand side of the equation in the statement of theorem.
\end{proof}

We also have the following property that will be important for the computation of the DR hierarchy for the extended $2$-spin theory in the next section.
\begin{lemma}\label{lemma:dilaton for DR hierarchy}
We have
$$
\frac{\d P^\alpha_{1,1}}{\d u^\beta}=D P^\alpha_{\beta,0},
$$
where $D:=\eps\frac{\d}{\d\eps}+\sum_{n\ge 0}u^\gamma_n\frac{\d}{\d u^\gamma_n}$.
\end{lemma}
\begin{proof}
The proof is again similar to the proof of the analogous equation for the Hamiltonian DR hierarchy (see \cite[Section~4.2.5]{Bur15a}) and based on the equation
\begin{gather*}
\frac{\d P^\alpha_{1,1}}{\d u^\beta}=\sum_{\substack{g,n\geq 0\\k_1,\ldots,k_n\geq 0}} \frac{\eps^{2g}}{n!} \Coef_{(a_1)^{k_1}\ldots(a_n)^{k_n}} \left(\int_{\DR_g(-\sum a_j,0,0,a_1,\ldots,a_n)} \hspace{-1.7cm}\lambda_g \psi_3 c_{g,n+3}(e^\alpha\otimes e_\beta\otimes e_1\otimes \otimes_{j=1}^n e_{\alpha_j}) \right)\prod_{j=1}^n u^{\alpha_j}_{k_j}
\end{gather*}
together with the property $\pi_*\psi_3=2g+n$, where $\pi\colon\oM_{g,n+3}\to\oM_{g,n+2}$ is the forgetful map, that forgets the third marked point.
\end{proof}

We end this section by remarking that several results and properties of the Hamiltonian DR hierarchy of a (partial) CohFT \cite{BR16} have an analogue for the DR hierarchy of an F-CohFT. In particular the string equation and recursion for the higher symmetries in terms of the flow~$\frac{\d}{\d t^1_1}$ can be proved for F-CohFTs, providing a far-reaching generalization of integrable systems of DR type \cite{BDGR16b} to the non-Hamiltonian context. This will be the topic of a forthcoming paper.

\subsection{DR hierarchy for the extended $2$-spin theory}

Consider the DR hierarchy for the extended $2$-spin theory:
\begin{gather}\label{eq:DR hierarchy for the extended 2-spin}
\frac{\d u^\alpha}{\d t^\beta_d}=\d_xP^\alpha_{\beta,d}(u^1,u^2,\eps),\quad 1\le \alpha,\beta\le 2,\quad d\ge 0.
\end{gather}
\begin{theorem}\label{theorem:DR hierarchy}
The DR hierarchy~\eqref{eq:DR hierarchy for the extended 2-spin} is related to the extended discrete KdV hierarchy by the Miura transformation
$$
u(u^1,u^2,\eps)=u^1,\qquad v(u^1,u^2,\eps)=\sqrt{R}u^2.
$$
\end{theorem}
\begin{proof}
In Section~\ref{subsubsection:step 5} we proved that the extended discrete KdV hierarchy is related to the hierarchy of topological type by the Miura transformation~\eqref{eq:Miura between DZ and discrete KdV}. Therefore, we have to prove that the DR hierarchy is related to the hierarchy of topological type by the Miura transformation
\begin{gather*}
w^1(u^1,u^2,\eps)=Lu^1,\qquad w^2(u^1,u^2,\eps)=L\sqrt{R}u^2.
\end{gather*}

Because of formula~\eqref{eq:extended and lambda_g}, we have $P^1_{2,d}=0$ and $P^1_{1,d}=\frac{(u^1)^{d+1}}{(d+1)!}$. Let $\odeg u^1_n:=2$ and $\odeg u^2_n:=1$, then from the definition of the DR hierarchy and the degree formula~\eqref{eq:degree of the extended class} it follows that $\odeg P^2_{\beta,d}=2d+\beta$. Consider the decomposition
$$
P^2_{\beta,d}=\sum_{k=0}^{d+\left[\frac{\beta}{2}\right]}P^2_{\beta,d,k},\quad \odeg_{u^1}P^2_{\beta,d,k}=2k,\quad \odeg_{u^2}P^2_{\beta,d,k}=2d+\beta-2k. 
$$
We see that the flows $\frac{\d}{\d t^2_0}$ and $\frac{\d}{\d t^1_1}$ of the DR hierarchy have the following form:
\begin{align}
&\frac{\d u^1}{\d t^2_0}=0, && \frac{\d u^1}{\d t^1_1}=u^1u^1_x,\label{eq:DR hierarchy, general 1}\\
&\frac{\d u^2}{\d t^2_0}=\d_x\left(P^2_{2,0,0}+T u^1\right), && \frac{\d u^2}{\d t^1_1}=\d_x\left(P^2_{1,1,0}+P^2_{1,1,1}\right),\label{eq:DR hierarchy, general 2}
\end{align}
where $T$ is a differential operator of the form
\begin{gather*}
T=1+\sum_{g\ge 1}T_g(\eps\d_x)^{2g},\quad T_g\in\mbC.
\end{gather*}

Clearly, the Miura transformation 
\begin{gather}\label{eq:DR--(u,w) Miura}
u^1(u,w,\eps)=u,\qquad u^2(u,w,\eps)=Tw
\end{gather}
transforms system~\eqref{eq:DR hierarchy, general 1},~\eqref{eq:DR hierarchy, general 2} to a system of the form~\eqref{eq:general system,1},~\eqref{eq:general system,2}, satisfying properties~\eqref{eq:general property,1}--\eqref{eq:general property,3}. Therefore, we can apply Proposition~\ref{proposition:main proposition about compatibility} and conclude that the differential polynomials~$Q$,~$C$ and~$B$ are given by formulas~\eqref{eq:formula for Q}--\eqref{eq:formula for B}. Let us now check that the unknown complex parameter $\theta$ from Proposition~\ref{proposition:main proposition about compatibility} is equal to~$\pm 1$.

\begin{lemma}\label{lemma:genus 1 of DR}
We have
\begin{gather*}
T_1=\frac{1}{24},\qquad \Coef_{\eps^2}P^2_{2,0,0}=-\frac{1}{48}\left(2u^2u^2_{xx}+(u^2_x)^2\right).
\end{gather*}
\end{lemma}
\begin{proof}
We compute
\begin{align*}
T_1=&\Coef_{a^2}\int_{\DR_1(a,-a,0)}\lambda_1 c^{2,\ext}_{1,3}(e^2\otimes e_1\otimes e_2)=\int_{\oM_{1,2}}\lambda_1 c^{2,\ext}_{1,2}(e^2\otimes e_2)=\\
=&\frac{1}{24}\int_{\oM_{1,2}}\iota^\circ_*(1) c^{2,\ext}_{1,2}(e^2\otimes e_2)\stackrel{\text{eq.~\eqref{eq:loop property for extended}}}{=}-\frac{1}{12}\int_{\oM_{0,4}}c^{2,\ext}_{0,4}(e^2\otimes e_2^{\otimes 3})=\frac{1}{24},
\end{align*}
where we denote by $\iota^\circ\colon\oM_{0,4}\to\oM_{1,2}$ the gluing map.

For the second equation of the lemma we do the following computation:
\begin{align*}
&\int_{\DR_1(a_1,a_2,a_3,0)}\lambda_1c^{2,\ext}_{1,4}(e^2\otimes e_2^{\otimes 3})\stackrel{\text{by Hain's formula}}{=}\\
=&\frac{\sum a_i^2}{2}\left[\int_{\oM_{1,4}}\lambda_1\psi_1 c_{1,4}^{2,\ext}(e^2\otimes e_2^{\otimes 3})-\left(\int_{\oM_{1,2}}\lambda_1 c^{2,\ext}_{1,2}(e^2\otimes e_2)\right)\left(\int_{\oM_{0,4}}c^{2,\ext}_{0,4}(e^2\otimes e_2^{\otimes 3})\right)\right]=\\
=&-\frac{a_2^2+a_3^2+a_2a_3}{24},
\end{align*}
which gives
\begin{align*}
\left.\Coef_{\eps^2}P^2_{2,0,0}\right|_{u^2_k=\sum_{n\in\mbZ}(in)^k p^2_n e^{inx}}=&-\frac{1}{2}\sum_{a_2,a_3\in\mbZ}\left(-\frac{a_2^2+a_3^2+a_2a_3}{24}\right)p^2_{a_2}p^2_{a_3}e^{i(a_2+a_3)x}=\\
=&\left.-\frac{1}{48}\left(2u^2u^2_{xx}+(u^2_x)^2\right)\right|_{u^2_k=\sum_{n\in\mbZ}(in)^k p^2_n e^{inx}}.
\end{align*}
\end{proof}

By this lemma,
\begin{gather*}
T=1+\frac{\eps^2}{24}\d_x^2+O(\eps^4),\qquad P^2_{2,0,0}=-\frac{1}{4}(u^2)^2-\frac{\eps^2}{48}\left(2u^2u^2_{xx}+(u^2_x)^2\right)+O(\eps^4),
\end{gather*}
which implies that $Q=-\frac{1}{4}w^2-\frac{\eps^2}{24}ww_{xx}+O(\eps^4)$. Therefore, $\theta^2=1$.

Let us now prove that $T=\sqrt{R}$. By Lemma~\ref{lemma:dilaton for DR hierarchy}, $\frac{\d P^2_{1,1}}{\d u^2}=D P^2_{2,0}$, and, in particular, 
$$
\frac{\d P^2_{1,1,1}}{\d u^2}=D T u^1.
$$
Since
$$
P^2_{1,1,1}=\frac{1}{2}T\left(u^1\cdot T^{-1}u^2\right)+\frac{1}{2}TR^{-1}\left(u^1\cdot RT^{-1}u^2\right)+\frac{\eps^2}{8}T\left(\d_x^2Ru^1\cdot RT^{-1}u^2\right),
$$
we get the relation
\begin{gather*}
\frac{1}{2}Tu^1+\frac{1}{2}TR^{-1}u^1+\frac{(\eps\d_x)^2}{8}TRu^1=DTu^1,
\end{gather*}
which can be equivalently written as
$$
\frac{d\hT}{d z}=\frac{1}{z}\left(-\frac{1}{2}+\frac{1}{2}\hR^{-1}+\frac{z^2}{8}\hR\right)\hT.
$$
This ordinary differential equation for the formal power series $\hT(z)\in\mbC[[z]]$ has a unique solution, satisfying the initial condition $\hT(0)=1$, and one can quickly check that the function~$\sqrt{\hR}$ satisfies this equation. Thus, $T=\sqrt{R}$.

We see that after the Miura transformations $w^1(u,w,\eps)=Lu$, $w^2(u,w,\eps)=RLw$ and~\eqref{eq:DR--(u,w) Miura} the flow $\frac{\d}{\d t^2_0}$ of both the hierarchy of topological type and the DR hierarchy has the form
\begin{align*}
\frac{\d u}{\d t^2_0}=&0,\\
\frac{\d w}{\d t^2_0}=&-\frac{1}{4}\d_x\left((Rw)^2-4u\right).
\end{align*}
For both hierarchies we have $\frac{\d u}{\d t^1_d}=\d_x\frac{u^{d+1}}{(d+1)!}$ and $\frac{\d u}{\d t^2_d}=0$. The dispersionless parts of the hierarchies coincide and $\odeg Q^2_{\beta,d}=\odeg P^2_{\beta,d}=2d+\beta$. Therefore, by Lemma~\ref{lemma:second uniqueness lemma}, the hierarchy of topological type and the DR hierarchy coincide in the coordinates $u$, $w$. This completes the proof of the theorem.
\end{proof}

%%%%%%%%%%%%%%%%%%%%%%%%%%%%%%%%%%%%%%%%%%%%%%%%%%%%%%%%%%%%%%%%%
%%%%%%%%%%%%%%%%%%%%%%%%%%%%%%%%%%%%%%%%%%%%%%%%%%%%%%%%%%%%%%%%%

\section{Extended $2$-spin theory and open Hodge integrals}

In~\cite{PST14,ST,Tes15} the authors initiated the study of the intersection theory on the moduli space $\oM_{g,k,l}$ of Riemann surfaces with boundary of genus $g$ with $k$ boundary marked points and~$l$ internal marked points. Recall that a closed Riemann surface is not considered as a Riemann surface with boundary and the genus of a Riemann surface with boundary is defined as the genus of its double. Moreover, in these works the authors constructed the integrals
\begin{gather}\label{eq:open psi-integrals}
\int_{\oM_{g,k,l}}\psi_1^{d_1}\psi_2^{d_2}\ldots\psi_l^{d_l}
\end{gather}
of the monomials in the psi-classes, attached to the internal marked points.

In~\cite[Section~5.2]{BCT17} the authors observed the following relation, which they called the {\it open-closed correspondence}:
\begin{gather}\label{eq:open-closed correspondence}
\int_{\oM_{0,k,l}}\prod_{i=1}^l\psi_i^{d_i}=(2\sqrt{-1})^{k-1}\int_{\oM_{0,l+k+1}}c^{2,\ext}_{0,l+k+1}(e^2\otimes e_1^{\otimes l}\otimes e_2^{\otimes k})\prod_{i=1}^l\psi_{i+1}^{d_i}.
\end{gather}
The authors of~\cite{BCT17} also proposed the idea that there should be a higher genus generalization of this correspondence, where the insertion $e^2$ corresponds to a boundary component of a Riemann surface with boundary and the insertion $e_2$ corresponds to a boundary marked point. To be more precise, the intersection number with a genus $g$ generalization of the class $c^{2,\ext}_{0,l+k+1}(e^2\otimes e_1^{\otimes l}\otimes e_2^{\otimes k})$ should correspond to the intersection number on the moduli space of Riemann surfaces with boundary obtained by removing an open disk from a closed Riemann surface of genus $g$. Note that the genus of such Riemann surfaces with boundary is equal to $2g$.

One of the original motivations of our work was to try to generalize the open-closed correspondence~\eqref{eq:open-closed correspondence} to all genera and construct an F-CohFT $\tc^{2,\ext}_{g,n+1}\colon (V^\ext)^*\otimes (V^\ext)^{\otimes n}\to H^{\even}(\oM_{g,n+1},\mbC)$ such that
\begin{align}
&\tc^{2,\ext}_{0,n+1}=c^{2,\ext}_{0,n+1},\label{eq:wrong F-CohFT,property 1}\\
&\tc^{2,\ext}_{g,n+1}(e^1\otimes e_1^{\otimes l}\otimes e_2^{\otimes k})=\delta_{k,0},\label{eq:wrong F-CohFT,property 2}\\
&\int_{\oM_{2g,k,l,1}}\prod_{i=1}^l\psi_i^{d_i}=C(g,k)\int_{\oM_{g,l+k+1}}\tc^{2,\ext}_{g,l+k+1}(e^2\otimes e_1^{\otimes l}\otimes e_2^{\otimes k})\prod_{i=1}^l\psi_{i+1}^{d_i},\label{eq:wrong F-CohFT,property 3}
\end{align}
where $C(g,k)$ are some rational constants, depending only on $g$ and $k$. Here $\oM_{2g,k,l,1}$ is the moduli space of Riemann surfaces with boundary with exactly one boundary component (see~\cite{ABT17}).

Surprisingly, there is a simple argument showing that an F-CohFT, satisfying properties~\eqref{eq:wrong F-CohFT,property 1}--\eqref{eq:wrong F-CohFT,property 3} doesn't exist. Indeed, suppose it exists and let us then consider the associated DR hierarchy. Property~\eqref{eq:wrong F-CohFT,property 3} dictates that
$$
\deg\tc^{2,\ext}_{g,l+k+1}(e^2\otimes e_1^{\otimes l}\otimes e_2^{\otimes k})=k-1,\quad\text{if $k$ is odd},
$$ 
and $\tc^{2,\ext}_{g,l+k+1}(e^2\otimes e_1^{\otimes l}\otimes e_2^{\otimes k})=0$, if $k$ is even. This implies that the flow $\frac{\d}{\d t^1_1}$ of the DR hierarchy has the following form:
\begin{align*}
&\frac{\d u^1}{\d t^1_1}=\d_x\left(\frac{(u^1)^2}{2}+\frac{\eps^2}{24}u^1_{xx}\right),\\
&\frac{\d u^2}{\d t^1_1}=\d_x\left(-\frac{(u^2)^3}{6}+u^1u^2+\alpha\eps^2 u^2_{xx}\right),
\end{align*}
for some complex constant $\alpha$. Using Lemma~\ref{lemma:dilaton for DR hierarchy}, we can then compute the equations for the flow $\frac{\d}{\d t^2_0}$:
\begin{align*}
&\frac{\d u^1}{\d t^2_0}=0,\\
&\frac{\d u^2}{\d t^2_0}=\d_x\left(-\frac{(u^2)^2}{4}+u^1\right).
\end{align*}
One can easily check that such flows $\frac{\d}{\d t^1_1}$ and $\frac{\d}{\d t^2_0}$ don't commute for any value of $\alpha$.

Nevertheless, the fact that $c^{2,\ext}_{g,k+l+1}(e^1\otimes e_1^{\otimes l}\otimes e_2^{\otimes k})=\delta_{k,0}\lambda_g$ motivates us to conjecture that the intersection numbers with the class $c^{2,\ext}_{g,k+l+1}(e^2\otimes e_1^{\otimes l}\otimes e_2^{\otimes k})$ should correspond to the open Hodge integrals. To be more precise, we propose the following conjecture.

\begin{conjecture}
There exists a geometric construction of open $\lambda_{2g}$-integrals
$$
\int_{\oM_{2g,k,l,1}}\lambda_{2g}\prod_{i=1}^l\psi_i^{d_i},
$$
such that the following relation holds:
$$
\int_{\oM_{2g,k,l,1}}\lambda_{2g}\prod_{i=1}^l\psi_i^{d_i}=C(g,k)\int_{\oM_{g,l+k+1}}c^{2,\ext}_{g,l+k+1}(e^2\otimes e_1^{\otimes l}\otimes e_2^{\otimes k})\prod_{i=1}^l\psi_{i+1}^{d_i},
$$
for some rational constants $C(g,k)$, depending on $g$ and $k$.
\end{conjecture}

\end{document}